\input ./style/arxiv-general.cfg
\documentclass[aop,MSNbibl,nameyear,dvips]{arximspdf}
\makeatletter
   \@ifpackageloaded{graphicx}{}{\usepackage{graphicx}}
\makeatother
\usepackage{mathrsfs}

\doi{10.1214/14-AOP956} 
\volume{43}
\issue{6}
\pubyear{2015}
\firstpage{3133}
\lastpage{3176}
\docsubty{FLA}

\makeatletter
\newcommand{\MD}{MD}
\newcommand{\rrvert}{\vert}
\newcommand{\llvert}{\vert}
\renewcommand{\mid}{|}
\newcommand{\R}{\mathbb{R}}
\newcommand{\F}{\mathcal{F}}
\newcommand{\N}{\mathbb{N}}
\newcommand{\G}{\mathcal{G}}
\newcommand{\E}{\mathbb{E}}
\newcommand{\1}{\mathbf{1}}
\newcommand{\Q}{\mathbb{Q}}
\newcommand{\X}{\mathcal{X}}
\newcommand{\Y}{\mathcal{Y}}
\newcommand{\B}{\mathcal{B}}
\newcommand{\dd}{\mathrm{d}}
\newcommand{\ba}{\mathrm{ba}}

\newtheorem{teo}{Theorem}[section]
\newtheorem{lem}[teo]{Lemma}
\newtheorem{cor}[teo]{Corollary}
\newtheorem{prop}[teo]{Proposition}
\newproclaim{defn}[teo]{Definition}
\newproclaim{ass}{Assumption}
\newproclaim{ex}[teo]{Example}
\newproclaim{rmk}[teo]{Remark}
\makeatother

\begin{document}
\begin{frontmatter}

\title{Supermartingales as Radon--Nikodym densities and related
measure extensions}
\runtitle{Supermartingales as Radon--Nikodym densities}

\begin{aug}
\author[A]{\fnms{Nicolas}~\snm{Perkowski}\thanksref{T1}\ead[label=e1]{perkowski@ceremade.dauphine.fr}}
\and
\author[B]{\fnms{Johannes}~\snm{Ruf}\corref{}\thanksref{T2}\ead[label=e2]{j.ruf@ucl.ac.uk}}
\runauthor{N. Perkowski and J. Ruf}
\affiliation{Universit\'e Paris Dauphine and University College London}
\address[A]{CEREMADE---UMR CNRS 7534\\
Universit\'e Paris Dauphine\\
Place du Mar\'echal De Lattre De Tassigny\\
75775 Paris Cedex 16\\
France\\
\printead{e1}}
\address[B]{Department of Mathematics\\
University College London\\
Gower Street\\
London WC1E 6BT\\
United Kingdom\\
\printead{e2}}
\end{aug}
\thankstext{T1}{Supported by the Fondation Sciences Math\'ematiques de
Paris (FSMP) and by
a public grant overseen by the French National Research Agency (ANR) as
part of the ``Investissements d'Avenir''
program (reference: ANR-10-LABX-0098), and acknowledges generous
support from Humboldt-Universit\"at zu Berlin, where a major part of
this work was completed.}
\thankstext{T2}{Supported from the Oxford-Man Institute of
Quantitative Finance at the University of Oxford, where a major part of
this work was completed.}

\received{\smonth{9} \syear{2013}}
\revised{\smonth{7} \syear{2014}}

%
\begin{abstract}
Certain countably and finitely additive measures can be associated to a
given nonnegative supermartingale.
Under weak assumptions on the underlying probability space, existence
and (non)unique\-ness results for such measures are proven.
\end{abstract}

%
\end{frontmatter}

\section{Introduction}\label{sec1}

It is a simple but very useful observation that a probability measure
$Q$ which is not absolutely continuous with respect to some reference
measure $P$
has a nonnegative $P$-supermartingale as its ``Radon--Nikodym
derivative.'' For instance, such supermartingales appear naturally in
the generalization of Girsanov's theorem to measures without absolute
continuity relation as in~\citet{Yoeurp1985}, or when working
with killed diffusions.

Conversely, given a nonnegative supermartingale, under suitable
assumptions on the probability space, it is possible to reconstruct a
measure associated to it, the so-called F\"ollmer measure. The behavior
of the F\"ollmer measure characterizes the most important properties of
the supermartingale; see \citeauthor{F1972} (\citeyear{F1972,F1973}); see also
\citeauthor{RufNovikov} (\citeyear{RufNovikov,Rufmartingale}), \citet{Cuidissertation}
and \citet{RufLarsson} for applications in the detection of
strict local martingales. Further applications include, among others,
potential theory [\citet{F1972,Airault1974}], simple proofs of
the main semimartingale decomposition theorems [\citet{F1973}],
filtration enlargements [\citet{Yoeurp1985,Kardarastimes}],
filtration shrinkage
[\citet{FoellmerProtter2010,Larsson2013}] and a simple approach
to the study of conditioned measures
[\citet{PerkowskiRuf}].

Measures associated to nonnegative supermartingales have also appeared
naturally in the duality approach to stochastic control. The dual
formulation has been developed for several important applications, such
as utility maximization [see, among many others, \citet
{KLSX,KramSchach1999,FollmerGundel}] or super-replication of contingent
claims [see, e.g., \citet{Jacka,Rufots}]. In many situations, the
dual variables represent nonnegative supermartingales. As observed by
\citet{KLSX} [see also \citet{KramSchach1999}], only in
very special situations do those supermartingales turn out to be
martingales, in which case computations of the dual problem can be
simplified via standard changes of measure.

In order to tackle the general case,
more general changes of measure have been suggested. There are mainly
two approaches, one relying on powerful arguments in functional
analysis, and the other one relying on deep probabilistic insights. The
functional analytic arguments identify supermartingales with the
elements of the dual space of bounded measurable functions, the space
of finitely additive measures; see, for example, \citet{CSW2001}
or \citet{Karatzas2003}, but also \citet{LarZit2013}. The
probabilistic approach, on the contrary, relies on certain canonical
assumptions on the underlying probability space but allows the
identification of supermartingales with countably additive probability
measures; see, for example, \citet{FollmerGundel}.


Recently, there has been an increased interest in economically
meaningful asset price models 
in which local martingales and supermartingales and their
interpretation as putative changes of measure appear naturally.
For example, there are models that allow for certain arbitrage
opportunities but have associated to them a class of dual
supermartingales [\citet{Pl2006,PH,KK,RufRunggaldier}]. The dual
supermartingales then correspond either to weakly equivalent finitely
additive local martingale measures [\citet{Kardarasfinitely}] or
to dominating local martingale measures, which turn out to be the
appropriate pricing operators in this context [\citet
{FK,Rufhedging,ImkellerPerkowski}]. Furthermore, underlying asset
prices have been modeled as strict local martingales under a pricing
measure and the corresponding associated measures have been constructed
in order to model certain phenomena, such as bubbles [\citet
{PP,KKNlocal}] or explosive exchange rates [\citeauthor{CFR2011} (\citeyear{CFR2011,CFRqnv})], and in order to compute actual quantities in such
models. Short-selling constraints lead directly to models in which
asset prices follow supermartingale dynamics [\citet{Pulido2013}]
and changes of num\'eraires in such models correspond to
supermartingales as Radon--Nikodym derivatives.

It is thus of great interest to construct the measure associated to a
given supermartingale $Z$. There are several different constructions
that all require different assumptions, and some of which only work on
an extended probability space:
\begin{itemize}
\item For a general filtered probability space $(\Omega,\F, (\F
_t)_{t \ge0}, P)$ with a nonnegative supermartingale $Z$, it is
possible to construct a finitely additive measure on $(\Omega\times
[0,\infty],\mathscr{A})$, where $\mathscr{A} \subset\mathscr{P}$
is a suitable algebra, and where $\mathscr{P}$ denotes the predictable
sigma algebra; see \citet{Metivier1975}. Without further
assumptions on $(\Omega, \F, (\F_t)_{t \ge0})$, this measure can be
extended to a countably additive measure on $(\Omega\times[0,\infty
], \mathscr{P})$ if and only if $Z$ is of class ($D$), in which case
one obtains the measure of \citet{Doleans1968}; see also \citet{M}.
\item Under certain topological assumptions on the filtered probability
space $(\Omega, \F, (\F_t)_{t \ge0}, P)$, it is possible to
construct the F\"ollmer measure on the enlarged space $(\Omega\times
[0,\infty],\mathscr{P})$.
We refer to \citet{F1972}, \citet{M} and \citet
{Stricker1975} for three different constructions. Under appropriate
conditions, it is also possible to construct the F\"ollmer measure on
$(\Omega, \F_{\zeta-})$, where $\zeta$ is a certain stopping time,
and not on an enlarged space~[\citet{Moy1953,M,Azema1976,DSBessel}].

\item Taking a different approach, if $Z$ is the pointwise limit of a
family of uniformly integrable martingales, then there exists a
finitely additive measure associated to it [\citet
{CSW2001,Karatzas2003}]. However, so far it seems to be not very well
understood under which conditions the supermartingale $Z$ is such a
pointwise limit of martingales.
\end{itemize}

In this work, we contrast the last two approaches of associating
countably and finitely additive measures to supermartingales. In the
countably additive case, we prove the existence of a probability
measure such that a given supermartingale $Z$ can be interpreted as
Radon--Nikodym density of this measure. In particular, we show that
this probability measure can already be constructed on the canonical
space itself and not only on the product space, even if the
supermartingale $Z$ is not a local martingale. Moreover, we provide
precise necessary and sufficient conditions for the uniqueness of such
a probability measure associated to the supermartingale $Z$.

In the finitely additive case, we show the existence of a finitely
additive measure associated to the supermartingale $Z$, as long as the
underlying filtered probability space is sufficiently rich, that is, as
long as it supports a Brownian motion. Furthermore, we show that such a
finitely additive measure is never unique. The argument for the
existence of such finitely additive measures also yields the existence
of uniformly integrable martingales that Fatou converge to a given
supermartingale.

\subsection*{Structure of the paper}
The paper is structured as follows: we conclude the \hyperref[sec1]{Introduction} by a
short overview of the notation used in the following. Section~\ref{SFollmerMeasure} introduces the notion of the F\"ollmer measure
associated to a supermartingale. Section~\ref{Sresults} contains the
main results concerning existence and (non)uniqueness of F\"ollmer
measures. Sections~\ref{sfoellmeronpathspace} and \ref
{sfinitelyadditive} then consist of the proofs of those results and of
some pedagogical examples.

Appendix~\ref{Afiltration} reviews modifications of processes if the
filtration is not augmented by null sets.
Appendix~\ref{Amultiplicative} recalls results concerning the
multiplicative decomposition of a supermartingale, which will be used
in the proofs of Sections~\ref{sfoellmeronpathspace} and \ref
{sfinitelyadditive}.
Appendix~\ref{Ameasure} provides a collection of definitions and
results concerning relevant measure-theoretic spaces. Appendix~\ref
{Aextension} discusses results concerning the extension of measures and
Appendix~\ref{Apathspace} lists important properties of the canonical
path space.

Appendix~\ref{AFatou} shows how the approximation techniques used to
prove the existence of the F\"ollmer finitely additive measure can be
modified to show that any supermartingale can be approximated, in the
sense of Fatou, by uniformly integrable martingales, provided the
underlying probability space supports a Brownian motion.
Appendix~\ref{AE} extends the discussion of nonuniqueness of the F\"
ollmer finitely additive measure and illustrates that the uniqueness
claim of the Carath\'eodory extension theorem does not hold among the
class of finitely additive measures. Finally, Appendix~\ref
{Aalternative} provides an alternative and insightful proof of one
lemma concerning the nonuniqueness of the F\"ollmer finitely additive measure.

\subsection*{Basic notation}
We shall use the convention that $\inf\varnothing= \infty$ and $\infty
\times\1_A(\omega) = 0$ for all $\omega\notin A$, where $A$ denotes
some event.
Expectations under a probability measure $R$ are denoted by $\E
_R[\cdot]$. Equalities between processes are to be \mbox{understood} up to
indistinguishability, and statements such as ``$G$ is a c\`adl\`ag
process'' or ``$G$ is nonnegative'' mean that these properties hold
almost surely---unless explicitly stated otherwise. We shall assume
that all considered semimartingales are (almost surely) c\`adl\`ag. If
$G = (G_t)_{t \geq0}$ is a l\`ag process, then we denote by $G_- =
(G_{t-})_{t \geq0}$ its left limit process, that is, $G_{0-} = G_0$
and $G_{t-} = \limsup_{s \uparrow t} G_s \mathbf{1}_{\limsup_{s
\uparrow t} G_s< \infty} + 0\times\mathbf{1}_{\limsup_{s \uparrow
t} G_s = \infty}$ for\vspace*{1pt} all $t > 0$, and by $\Delta G = G - G_-$ the
jump process of the process $G$. Similarly, if $G$ is a l\`ad process,
we denote by $G_+ = (G_{t+})_{t \geq0}$ its right limit process.

Throughout this paper, we are given a filtered probability space
$(\Omega, \F, \break (\F_t)_{t \geq0}, P)$, where $\Omega$ denotes a
nonempty set, $\F$ a sigma algebra on $\Omega$, $(\F_t)_{t \geq0}$
a right-continuous filtration with $\F_t \subseteq\F$ for all $t \ge
0$, and
$P$ a probability measure. We set $\F_\infty= \bigvee_{t \geq0} \F
_t \subset\F$. Moreover, we are given a nonnegative, right-continuous
$P$-supermartingale $Z = (Z_t)_{t \geq0}$ with $E_P[Z_0] = 1$. We
stress the fact that $(\F_t)_{t \geq0}$ will not always be complete
with respect to $P$; see Appendix~\ref{Afiltration} for a discussion
of this critical point.
We shall always silently assume that the notions of martingales,
supermartingales, etc., hold with respect to the filtration $(\F_t)_{t
\geq0}$.

\section{F\"ollmer measures associated to a supermartingale}\label{sec2} \label{SFollmerMeasure}

We like to think of the $P$-supermartingale $Z$ as the
``Radon--Nikodym density'' of a probability measure $Q$ that is not
necessarily absolutely continuous with respect to $P$; also $P$ is not
necessarily absolutely continuous with respect to $Q$. Our aim is to
recover the measure $Q$.

On a general probability space, for example, one with a completed
filtration, such a probability measure $Q$ does not always exist on
$(\Omega, \F)$---except if $Z$ is a uniformly integrable
martingale. However, as we will show, if $(\Omega, \F, (\F_t)_{t
\geq0})$ is the space of (possibly explosive) right-continuous paths
with left limits along with the canonical filtration, such a
probability measure always exists.
Moreover, without such a canonical assumption, but under the assumption
that the filtered probability space $(\Omega, \F, (\F_t)_{t \geq0},
P)$ supports some Brownian motion, we shall see that it is still
possible to associate a finitely additive measure $Q$ to the
\mbox{$P$-}supermartingale $Z$.

In the following, we make precise the meaning of ``measures associated
to supermartingales.''

\subsection{F\"ollmer countably additive measures}\label{sec2.1}

Here, we explain in which way a supermartingale can be interpreted as
the Radon--Nikodym derivative of a countably additive probability measure.
We begin with the following definition.

\begin{defn}
If $Q$ and $\tau$ are a probability measure and a stopping time on
$(\Omega, \F, (\F_t)_{t\ge0})$, then $(Q,\tau)$ is called a \emph{F\"ollmer pair} for $Z$ if
%
\begin{eqnarray}\label{ekunita-yoeurp}
P[\tau= \infty] &=& 1\quad\mbox{and}\nonumber
\\
Q\bigl[A \cap\{\rho< \tau\}\bigr] &=& \E_P[
Z_\rho\1_A ]
\\
\eqntext{\mbox{for all $A \in\F_\rho$
and finite stopping times $\rho$}.}
\end{eqnarray}
In that case, we also call $Q$ a \emph{F\"ollmer} (\emph{countably additive})
\emph{measure} for $(Z,\tau)$, or, slightly abusing notation, a \emph{F\"
ollmer} (\emph{countably additive}) \emph{measure} for $Z$.
\end{defn}

Note that every F\"ollmer measure $Q$ is defined on $(\Omega,\F)$,
despite the fact that~(\ref{ekunita-yoeurp}) only involves the
restriction of $Q$ to $\F_\infty\subset\F$. If $(Q, \tau)$ is a F\"
ollmer pair for the \mbox{$P$-}supermartingale $Z$ then the pair $(Z, \tau)$
is called the \emph{Kunita--Yoeurp decomposition} of $Q$ with respect
to $P$. In that case, the two measures $Q\mid _{\F_t}[\cdot\cap\{t \ge
\tau\}]$ and $P\mid _{\F_t}[\cdot]$ are singular for each $t \geq0$ as
the second one has full mass on the event $\{\tau= \infty\}$ while
the first one assigns measure zero to it. Hence, the
Kunita--Yoeurp decomposition can be interpreted as a progressive
Lebesgue decomposition on filtered probability spaces. The
decomposition was introduced by \citet{Kunita1976} for Markov
processes. The general formulation is due to \citet{Yoeurp1985},
who used it to prove a generalized Girsanov theorem for probability
measures that are not necessarily absolutely continuous with respect to
each other.

Given two probability measures $P$ and $Q$, where $Q$ has the
Kunita--Yoeurp decomposition $(Z,\tau)$ with respect to $P$, the
stopping time $\tau$ is uniquely determined up to a $Q$-null set, and
the $P$-supermartingale $Z$ is uniquely determined up to a
$P$-evanescent set; see Proposition~2 in \citet{Yoeurp1985}. As
Theorem~\ref{tfoellmercountablesummary} below yields, it might well be
possible, however, to associate two different F\"ollmer countably
additive measures to a given $P$-supermartingale $Z$.

%
\begin{defn}
We say that the F\"ollmer pair for $Z$ is \emph{unique} if given two
probability measures $Q$ and $\widetilde{Q}$ and two stopping times
$\tau$ and $\widetilde{\tau}$ such that $(Q,\tau)$ and $(\widetilde
{Q}, \widetilde{\tau})$ both satisfy (\ref{ekunita-yoeurp}), we have
$Q = \widetilde{Q}$ (and $Q[\tau= \widetilde{\tau}] = 1$).

If $\tau$ is a stopping time, then we say that the F\"ollmer
(countably additive) measure for $(Z,\tau)$ is \emph{unique} if,\vspace*{1pt}
given two probability measures $Q$ and $\widetilde{Q}$ such that
$(Q,\tau)$ and $(\widetilde{Q}, \tau)$ both satisfy (\ref
{ekunita-yoeurp}), we have $Q = \widetilde{Q}$.
\end{defn}

In order to verify whether a given probability measure is a F\"ollmer
countably additive measure for the $P$-supermartingale $Z$, it
suffices to verify (\ref{ekunita-yoeurp}) for deterministic times.

%
\begin{prop} \label{PP12}
If $Q$ and $\tau$ are a probability measure and a stopping time on
$(\Omega, \F, (\F_t)_{t\ge0})$ such that $
P[\tau= \infty] = 1$ and $Q[A \cap\{t < \tau\}] = \E_P[ Z_t \1_A]$
for all $t \ge0$ and all $A \in\F_t$,
then $(Q,\tau)$ is a \emph{F\"ollmer pair} for $Z$.
Moreover,
%
\begin{equation}
\label{ekunita-yoeurprv} \E_{Q}[G \1_{\{\rho< \tau\}}] = \E_P[Z_\rho
G \1_{\{\rho<\infty\}
\cap\{Z_\rho>0\}}]
\end{equation}
holds for all $[0,\infty]$-valued, $\F_\rho$-measurable random
variables $G$ and for all stopping times $\rho$; in particular, $
{Q}[\rho< \tau] = \E_P[Z_\rho]$ for all finite stopping times $\rho$.
\end{prop}

\begin{pf}
Using the linearity of the expectation operator and the monotone
convergence theorem, it is sufficient to show (\ref{ekunita-yoeurprv})
for a fixed stopping time $\rho$, with $G = \1_A$ for an arbitrary $A
\in\F_\rho$. If $\rho$ takes only countably many values, then this
identity follows directly. For the general case, consider the
nonincreasing sequence $(\rho_n)_{n \in\N}$ of stopping times where
$\rho_n = \inf\{k 2^{-n}\dvtx  k 2^{-n} \geq\rho, k \in\N\}$. Then
(\ref{ekunita-yoeurprv}) holds with $G$ replaced by $\1_A$ and with
$\rho$ replaced by $\rho_n$ for each $n \in\N$. Finally, by taking
limits on both sides and using the fact that the nonnegative,
discrete-time backward $P$-supermartingale $(Z_{\rho_n})_{n \in\N}$
is uniformly integrable, we may conclude.
\end{pf}

%
\begin{cor} \label{C24}
If $(Q, \tau)$ is a F\"ollmer pair for $Z$ then the following two
statements hold:
\begin{itemize}
\item$P\llvert  _{\F_{\infty}} \ll Q\rrvert  _{\F_{\infty}}$ if and only if
$P[\lim_{t \uparrow\infty} Z_t > 0] = 1$;
\item$Q\llvert  _{\F_{\infty}} \ll P\rrvert  _{\F_{\infty}}$ if and only if $\E
_P[\lim_{t \uparrow\infty} Z_t] = 1$.
\end{itemize}
\end{cor}

\begin{pf}
Let $A = \{\lim_{t \uparrow\infty} Z_t = 0\} \cap\{\tau= \infty\}
$ and note that $Q[A] = 0$. This and the fact that $P[\tau= \infty
]=1$ yield the first ``only if'' implication. For the reverse
direction, let $B \in\bigcup_{t \ge0} F_t$ and observe that
\[
Q[B] \ge\lim_{t \uparrow\infty} Q\bigl[B \cap\{ \tau> t\}\bigr] = \lim
_{t
\uparrow\infty} \E_P[1_B Z_t] \ge
\E_P \Bigl[1_B \lim_{t \uparrow
\infty}
Z_t \Bigr].
\]
By the monotone class theorem, this extends to all $B \in\F_\infty$. Since
\[
P \Bigl[\lim_{t \uparrow\infty}Z_t > 0 \Bigr] = 1
\]
by assumption, we deduce that $P\llvert  _{\F_{\infty}} \ll Q\rrvert  _{\F_{\infty
}}$. The second equivalence follows from Proposition~III.3.5
\citet{JacodS}.
\end{pf}

The following observation describes the dynamics of the
$P$-supermartingale $Z$ under the F\"ollmer measure.

\begin{prop}
If $(Q, \tau)$ is a F\"ollmer pair for $Z$ then the process $Y =
(Y_t)_{t \geq0}$, given by $Y_t = 1/Z_t \1_{\{t < \tau\}}$, is a
(nonnegative) $Q$-supermartingale. Moreover, the following two
statements hold:
\begin{itemize}
\item$Y$ is a $Q$-local martingale if and only if $Z$ does not jump
to zero under $P$;
\item$Y$ is a $Q$-martingale if and only if $P[Z_t > 0] = 1$ for all
$t \geq0$.
\end{itemize}
\end{prop}

\begin{pf}
The statement follows from (\ref{ekunita-yoeurprv}) and a version of
Bayes' rule; for details, see the proof of Theorem~2.1 in \citet{CFR2011}.
\end{pf}

\subsection{F\"ollmer finitely additive measures}\label{sec2.2}

Recall that $\ba(\Omega, \F)$ is the space of bounded, finitely
additive set functions on $\F$ that take their values in $\R$. An
element $Q \in\ba(\Omega,\F)$ is called a \emph{finitely additive
probability measure} if it is nonnegative and satisfies $Q[\Omega] =
1$. In that case, $Q$ can be uniquely decomposed into a regular part
$Q^r \ge0$ and a singular part $Q^s \ge0$; see Theorem~III.7.8 in
\citet{DunfordSchwartzI}. Here, $Q^r$ is a sigma-additive measure
on $(\Omega, \F)$, and $Q^s$ is purely finitely additive, that is,
any sigma-additive measure $\mu$ on $\F$ which satisfies $0 \le\mu
\le Q^s$ is constantly 0.

If $Q$ and $R$ are two finitely additive measures, then $Q$ is said to
be \emph{weakly absolutely continuous} with respect to $R$ if for all
$A \in\F$ we have that $R[A] = 0$ implies $Q[A] = 0$; see
Remark~6.1.2 in \citet{Rao1983}. We shall write $\ba(\Omega, \F, P)$ for the space of all finitely additive measures on $\F$ that are
weakly absolutely continuous with respect to $P$; we write $\ba
_1(\Omega, \F, P)$ for all nonnegative elements of $\ba(\Omega,\F,P)$ that have total mass one.

%
\begin{defn} \label{DFAFM}
A weakly absolutely continuous, finitely additive probability measure
$Q \in\ba_1(\Omega, \F,P)$, such that
%
\begin{equation}
\label{eqQSing} \qquad (Q\mid _{\F_\rho})^r[A] = \E_P[Z_\rho
\1_A]\qquad\mbox{for all $A \in \F_\rho$ and finite
stopping times $\rho$},
\end{equation}
is called \emph{F\"ollmer finitely additive measure} for $Z$.
\end{defn}

Recall that the dual space $L^\infty(\Omega, \F, P)^*$ of $L^\infty
= L^\infty(\Omega, \F, P)$ can be identified with the elements of
$\ba(\Omega, \F,P)$; see Theorem~IV.8.16 in \citet
{DunfordSchwartzI}. This is the reason why finitely additive
probability measures naturally appear in the dual approach to
stochastic optimization problems. For further details, see also
\citet{CSW2001} and \citet{Karatzas2003}.

In Example~\ref{excounterexampledeterministictimesimplystoppingtimesfinitelyadditive}
below, we construct a $P$-supermartingale $Z$ and two finitely
additive probability measures $Q_1, Q_2 \in\ba_1(\Omega, \F,P)$
that satisfy\break  $(Q_1\mid _{\F_t})^r = (Q_2\mid _{\F_t})^r$ for all $t \ge0$.
Moreover, $Q_1$ satisfies (\ref{eqQSing}), and there exists a finite
stopping time $\rho$ such that
\[
(Q_1\mid _{\F_\rho})^r = 0 \neq P\mid
_{\F_\rho} = (Q_2\mid_{\F_\rho})^r.
\]
Therefore, there is no result corresponding to Proposition~\ref{PP12}
in the finitely additive case. To wit, if a finitely additive measure
$Q$ satisfies (\ref{eqQSing}) for deterministic times, then this does
not automatically imply that $Q$ satisfies (\ref{eqQSing}).

\subsection{Comparison of F\"ollmer countably and finitely additive measures}\label{sec2.3}

We have introduced two different notions of F\"ollmer measures, and it
is natural to ask how these two concepts are related. As it turns out,
in most situations they are mutually exclusive, despite their apparent
similarity.

%
\begin{prop}
If $Z$ is a uniformly integrable $P$-martingale and if \mbox{$\F=\F_\infty
$}, then each F\"ollmer countably additive measure for $Z$ is a F\"
ollmer finitely additive measure for $Z$.
If $Z$ is not a uniformly integrable $P$-mar\-tin\-gale, then the sets
of F\"ollmer countably additive measures for $Z$ and of F\"ollmer
finitely additive measures for $Z$ are disjoint.
\end{prop}

\begin{pf}
The statement follows from the second equivalence in Corollary~\ref{C24}.
\end{pf}

We shall see in Theorem~\ref{Tapprox} below and also in Appendix~\ref
{AE} that, in the case of a uniformly integrable martingale $Z$, the
class of F\"ollmer finitely additive measures is strictly larger than
the class of F\"ollmer countably additive measures, as long as the
probability space is sufficiently rich. However, in general, the
existence of a F\"ollmer countably additive measure does not imply the
existence of a F\"ollmer finitely additive measure, nor does the
opposite implication hold.

%
\begin{ex}
Assume that $Z$ is a $P$-local martingale which is not a uniformly
integrable $P$-martingale
and assume that the filtration $(\F_t)_{t \ge0}$ is augmented by all
$P$-null sets in $\F$. In Example~\ref{exlocalmartingalefinitelyadditivemeasure} below, we show that there
exists a F\"ollmer finitely additive measure for 
$Z$. However, since $\E_P[Z_0] = 1$ holds, any F\"ollmer countably
additive measure $Q$ for $Z$ is absolutely continuous with respect to
$P$ on the sigma algebra $\F_0$. Since $\F_0$ contains all $P$-null
sets, 
$Q$ is absolutely continuous with respect to 
$P$, which is only possible if 
$Z$ is a uniformly integrable $P$-martingale. Thus, there exists no F\"
ollmer countably additive measure for the $P$-local martingale $Z$.

This example illustrates why we shall assume an incomplete filtration
when constructing F\"ollmer countably additive measures below. If the
$P$-supermartingale $Z$ is a martingale, we could also assume a
filtration that is enlarged in a progressive manner; see \citet
{Bichteler2002} or \citet{Najnudel2011} for details. If $Z$ is a
local martingale, then it is still possible, by using a localization
sequence, to perform such a progressive enlargement; see \citet
{Kreher2013}; however, in that case the filtration depends on the local
martingale $Z$ itself. If $Z$ is only a $P$-supermartingale and not a
local martingale, then finding a progressive completion that would
allow us to construct F\"ollmer countably additive measures for $Z$
seems impossible. We continue this discussion on the issue of
completing the filtration in Remark~\ref{runiqueness} below.
\end{ex}

%
\begin{ex}
Assume that $\Omega= \{0,1\}$, $\F= \sigma(\{0\}) = \F_t$ for all
$t \ge1$, $\F_t = \{\varnothing, \Omega\}$ for all $t \in[0,1)$ and
$P[\{0\}]=1$. Moreover, assume that $Z$ satisfies $Z_t = \1_{t <1}$ for
all $t \geq0$.
Consider the probability measure $Q$ that satisfies $Q[\{1\}]=1$ and
the stopping time $\tau= 1+\infty\1_{\{0\}}$, which satisfies $P[\tau
=\infty] = 1 = Q[\tau= 1]$.
Then
$(Q,\tau)$ is a F\"ollmer pair for the $P$-supermartingale $Z$.
However, there is no F\"ollmer finitely additive probability measure
for the $P$-supermartingale $Z$ since $P$ is the only finitely
additive probability measure on $(\Omega,\F)$ which is weakly
absolutely continuous with respect to $P$.
\end{ex}

\section{Existence and (non)uniqueness}\label{sec3} \label{Sresults}
We next collect the main results of this paper concerning existence and
(non)uniqueness of F\"ollmer countably and finitely additive measures
associated to the nonnegative $P$-super-\break martingale $Z$.

\subsection{The countably additive case}\label{sec3.1}

In the countably additive case, we shall rely on a specific choice of a
canonical probability space. This motivates us to formulate the
following assumption (we recall the definition of several
measure-theoretic notions such as ``state space'' in Appendix~\ref{Ameasure}).

{\renewcommand{\theass}{($\mathcal{P}$)}
\begin{ass}\label{assP}
Let $E$ be a state space, and let $\Delta\notin E$ be a cemetery
state. For all $\omega\in(E\cup\{\Delta\})^{[0,\infty)}$ define
\[
\zeta(\omega) = \inf\bigl\{t \ge0\dvtx  \omega(t) = \Delta\bigr\}.
\]
Let $\Omega\subset(E\cup\{\Delta\})^{[0,\infty)}$ be the space of
paths $\omega\dvtx  [0,\infty) \rightarrow E \cup\{\Delta\}$, for which
$\omega$ is c\`adl\`ag on $[0, \zeta(\omega))$, and for which
$\omega(t) = \Delta$ for all $t \ge\zeta(\omega)$.
For all $t \ge0$ define $X_t(\omega) = \omega(t)$ and the sigma
algebras $\F^0_t = \sigma(X_s\dvtx  s \in[0,t])$ and $\F_t = \bigcap_{s
> t} \F^0_s$. Moreover, set $\F= \bigvee_{t \ge0} \F^0_t = \bigvee_{t \ge0} \F_t = \F_\infty$.
\end{ass}}%

Thus, under Assumption~\ref{assP}, the states of the world
$\omega\in\Omega$ are paths taking values in a state space up to a
certain time when they get absorbed in a cemetery state $\Delta$.
Before the time of absorption, those paths are assumed to be c\`adl\`
ag. The filtration $(\F_t)_{t \geq0}$ is the right-continuous
modification of the canonical filtration $(\F_t^0)_{t \geq0}$.

If $\rho$ is an $(\F_t)_{t \geq0}$-stopping time, then the sigma
algebra $\F_{\rho-}$ is defined as
%
\begin{equation}
\label{eFrho-definition} \F_{\rho-} = \F^0_0 \vee\sigma
\bigl( A \cap\{\rho> t\}\dvtx  A \in\F_t, t \ge0\bigr).
\end{equation}
For later use, note that $\rho$ is $\F_{\rho-}$-measurable.
This definition is slightly different from the usual one, where $\F
^0_0$ would be replaced by $\F_0$ in (\ref{eFrho-definition}). The
definition in (\ref{eFrho-definition}), taken from \citet
{F1972}, has the advantage that $\F_{\rho-}$ is countably generated,
as Lemma~\ref{lCanProps} will show, which collects several properties
of the probability space in Assumption~\ref{assP}.

We define the nondecreasing sequence $(\widehat{\tau}_n^Z)_{n \in\N
}$ of stopping times and the stopping time $\widehat{\tau}{}^Z$ by
%
\begin{equation}
\label{eqtauM} \widehat{\tau}_n^Z = \inf\{t \ge0\dvtx
Z_t \ge n\} \wedge n; \qquad \widehat{\tau}{}^Z = \lim
_{n \uparrow\infty} \widehat{\tau}_n^Z.
\end{equation}

Before we get to the (somewhat subtle) precise formulation of our
results, let us describe them informally. We show that, under
Assumption~\ref{assP}, it is possible to construct a F\"ollmer
countably additive measure for the $P$-supermartingale $Z$ on the
space $(\Omega,\F)$, as long as $\E_P[Z_\zeta\1_{\{\zeta<\infty\}
}] = 0$ or $Z$ is a local martingale. In particular, this is the case
if $P$ satisfies $P[\zeta< \infty] = 0$. Essentially, the F\"ollmer
countably additive measure of $Z$ is unique if $Z$ is a martingale (not
necessarily uniformly integrable), or if $Z$ is a local martingale
which explodes at time $\zeta$ and not before.

If the $P$-supermartingale $Z$ has a nontrivial part of finite
variation, then we have to artificially make $Q$ lose mass to obtain a
F\"ollmer countably additive measure for $Z$. Since we have a degree of
freedom here in choosing where to send the mass of $Q$, it is not
surprising that in this case we never have uniqueness---except
possibly if the state space $E$ is countable. Of course, if we fix the
stopping time $\tau$ in the F\"ollmer pair for the
$P$-supermartingale $Z$, and hereby implicitly specify where we send
the mass of $Q$, then it is also possible to have uniqueness if $Z$ is
a $P$-supermartingale. In particular, once the stopping time $\tau$
is fixed, $Q$ is always uniquely determined on $\F_{\tau-}$ and we
only have to study under which conditions there exists a unique
extension to $\F$.

%
\begin{teo}\label{tfoellmercountablesummary}
Under Assumption~\ref{assP}, suppose that one (or both) of the
following conditions hold:
\begin{itemize}
\item the $P$-supermartingale $Z$ is a $P$-local martingale;
\item the probability measure $P$ satisfies $\E_P[Z_\zeta\1_{\{\zeta
<\infty\}}] = 0$.
\end{itemize}
Then there exists a F\"ollmer pair $(Q,\tau)$ for $Z$. If the
$P$-supermartingale $Z$ is a $P$-local martingale, then we can take
$\widehat{\tau}{}^Z$, defined in (\ref{eqtauM}), as the stopping time;
to wit, in that case there exists a F\"ollmer countably additive
measure $\widehat{Q}{}^Z$ for $Z$ such that $(\widehat{Q}{}^Z, \widehat
{\tau}{}^Z)$ is a F\"ollmer pair.
Moreover, the following statements always hold:
\begin{longlist}[(III)]
\item[(I)] The following conditions are equivalent:
\begin{enumerate}[(a)]
\item[(a)] the set $\{\tau< \zeta\}$ is $Q\mid _{\F_{\tau-}}$-negligible;
\item[(b)] there is a unique F\"ollmer countably additive measure for
$(Z,\tau)$.
\end{enumerate}
\item[(II)] If $\widetilde{\tau}$ is a stopping time such that the
pair $(Q, \widetilde{\tau})$ also
satisfies~(\ref{ekunita-yoeurp}), then
$Q[\tau= \widetilde{\tau}] = 1$.
\item[(III)] The following statement in~\textup{(c)} always implies the one
in~\textup{(d)}. The reverse implication holds provided that the state space $E$
is uncountable.
\begin{enumerate}[(a)]
\item[(c)] The $P$-supermartingale $Z$ is a $P$-local martingale and
the set $\{\widehat{\tau}{}^Z < \zeta\}$ is $\widehat{Q}{}^Z\mid _{\F
_{\widehat{\tau}{}^Z -}}$-negligible;\vspace*{1pt}
\item[(d)] there is a unique F\"ollmer pair for $Z$.
\end{enumerate}
\end{longlist}
\end{teo}

%
\begin{rmk} \label{runiqueness}
\citet{Azema1976} also show the existence of a F\"ollmer countably
additive measure $Q$ for $(Z,\zeta)$ on $(\Omega, \F)$. Their
construction is quite different from the one presented below and does
not address the question of uniqueness: after fixing the stopping time
$\zeta$, there exists at most one probability measure $Q$ for which
$(Q,\zeta)$ is a F\"ollmer pair; see point~(I) in the previous theorem.

\citet{Azema1976} construct the countably additive F\"ollmer
measure directly on the universal completion $(\Omega, \F^u)$.
Indeed, if we construct $Q$ according to Theorem~\ref
{tfoellmercountablesummary} and then augment $\F_t$ with the
intersection of $P$- and $Q$-nullsets for all $t \geq0$ to obtain a
filtration $(\F^{P+Q}_t)_{t \geq0}$ and also a sigma algebra $\F
^{P+Q}$, the two probability measures $Q$ and $P$ can be uniquely
extended to $(\Omega, \F^{P+Q})$ by Lemma~\ref{lCanProps} and
Theorem~\ref{tlubinextension} in the Appendix~\ref{sec9}.
Thus, in particular, Theorem~\ref{tfoellmercountablesummary} also
yields the existence of a F\"ollmer countably additive measure on the
universally augmented space $(\Omega, \F^u)$. Note, however, that the
universally completed filtration $(\F^u_t)_{t \geq0}$ still misses
some of the nice properties of complete filtrations: for example, it is
not clear that supermartingales have identically c\`adl\`ag
modifications under $(\F^u_t)_{t\geq0}$.
\end{rmk}

The proof of the uniqueness statement in Theorem~\ref
{tfoellmercountablesummary} relies on the following observation.

%
\begin{lem}\label{lminimality}
Assume that $Z$ is a nonnegative $P$-local martingale. Then the F\"
ollmer pair $(\widehat{Q}{}^Z, \widehat{\tau}{}^Z)$ from Theorem~\ref
{tfoellmercountablesummary} is minimal in the following sense. If $(Q,
\tau)$ is another F\"ollmer pair then $Q\mid _{\F_{(\widehat{\tau}{}^Z
\vee\tau)-}}$ is uniquely determined by $Q\mid _{\F_{\tau-}}$, and
$Q[\widehat{\tau}{}^Z = \tau]=1$. In particular, we have ${Q}\mid _{\F
_{\widehat{\tau}{}^Z-}} = \widehat{Q}{}^Z\mid _{\F_{\widehat{\tau}{}^Z-}}$.
\end{lem}

Section~\ref{sfoellmeronpathspace} contains the proofs of Theorem~\ref
{tfoellmercountablesummary} and Lemma~\ref{lminimality}.

%
\begin{rmk}
The pair $(\widehat{Q}{}^Z, \widehat{\tau}{}^Z)$ of Lemma~\ref
{lminimality} is minimal in the sense of Lemma~\ref{lminimality}, but
usually not unique.
For example, consider the canonical probability space of
Assumption~\ref{assP} with $E = [0,\infty)$ equipped with two
measures $Q_1$ and $Q_2$ where $Q_1$ makes the canonical process a
Brownian motion stopped when hitting zero and $Q_2$ makes the canonical
process a Brownian motion killed when hitting zero; that is, if $\rho
_1$ denotes the first hitting time of zero by the canonical process,
then $Q_1[\rho_1<\infty] = 1 = Q_2[\zeta< \infty]$ and $Q_1[\zeta<
\infty] = 0 = Q_2[\rho_1 < \infty]$.

Now, if the canonical process is a three-dimensional Bessel process
under the probability measure $P$ and if $Z$ denotes its reciprocal,
then it is easily verified that both $(Q_1,
\rho_1)$ and $(Q_2, \zeta)$ satisfy (\ref{ekunita-yoeurp}).
However, those two pairs clearly do not agree. The minimal pair
$(\widehat{Q}{}^Z, \widehat{\tau}{}^Z)$ of Lemma~\ref{lminimality},
where $\widehat{Q}{}^Z$ is a-priori only defined on $\F_{\widehat{\tau
}{}^Z-}$, can be extended to $\F$ either by $Q_1$ or $Q_2$ (or other measures).
\end{rmk}

The following proposition provides an important sufficient criterion
for the uniqueness of the F\"ollmer countably additive measure in
Theorem~\ref{tfoellmercountablesummary}.

%
\begin{prop}
In the\vspace*{1pt} setup of Theorem~\ref{tfoellmercountablesummary}, if the
nonnegative \mbox{$P$-}\break super\-martingale $Z$ is a $P$-martingale, then
$\widehat{Q}{}^Z[\widehat{\tau}{}^Z= \infty] = 1$; in particular, then
the set $\{\widehat{\tau}{}^Z< \zeta\}$ is $\widehat{Q}{}^Z\mid _{\F
_{\widehat{\tau}{}^Z-}}$-negligible and there is a unique F\"ollmer
pair for $Z$.
\end{prop}

\begin{pf}
If $Z$ is a $P$-martingale, then
\begin{eqnarray*}
\widehat{Q}{}^Z\bigl[\widehat{\tau}{}^Z < \infty\bigr] &=&
\lim_{n \uparrow
\infty} \widehat{Q}{}^Z \bigl[\widehat{
\tau}{}^Z \le n \bigr] = \lim_{n \uparrow\infty} \bigl(1 -
\widehat{Q}{}^Z \bigl[\widehat{\tau }{}^Z > n \bigr] \bigr)
\\
&=& \lim_{n \uparrow\infty} \bigl(1 - \E_P[Z_n]
\bigr) = 0,
\end{eqnarray*}
which completes the proof.
\end{pf}

The next result contains a discussion of the missing implication from
(d) to (c) in Theorem~\ref{tfoellmercountablesummary} if the state
space $E$ is countable.

\begin{prop} \label{PECountable}
Under Assumption~\ref{assP}, suppose that the state space $E$ is
countable. Then we can distinguish the following cases:
\begin{longlist}[(B)]
\item[(A)]
If $E$ has exactly one element and if $P[\zeta<\infty] = 0$ then
there is a unique F\"ollmer pair for each nonnegative
$P$-supermartingale $Z$. However, if \mbox{$P[\zeta<\infty] >0$} then the
F\"ollmer pair is not necessarily unique.
\item[(B)] If $E$ has more than one element, then:
\begin{enumerate}[(ii)]
\item[(i)] there exists a probability measure $P$ on the sigma algebra
$\F$, with $P[\zeta<\infty] = 0$, such that for each
$P$-supermartingale $Z$ that is not a $P$-local martingale there are
at least two different F\"ollmer pairs for $Z$;
\item[(ii)] there exists a probability measure $P$ on the sigma
algebra $\F$, with $P[\zeta<\infty] = 0$, and a $P$-supermartingale
$Z$ that is not a $P$-local martingale such that there is a unique F\"
ollmer pair for $Z$.
\end{enumerate}
\end{longlist}
\end{prop}

The proof of Proposition~\ref{PECountable} can be found in
Section~\ref{sfoellmeronpathspace}. The proposition implies, in
particular, that the implication from (d) to (c) in Theorem \ref
{tfoellmercountablesummary} requires $E$ to be uncountable.

\subsection{The finitely additive case}\label{sec3.2}

In the finitely additive case, we assume that the underlying
probability space is sufficiently large to support a Brownian motion.
If that assumption holds then it is possible to associate a F\"ollmer
finitely additive measure to any nonnegative $P$-supermartingale.

{\renewcommand{\theass}{($\mathcal{B}$)}
%
\begin{ass}\label{assB}
The filtered probability space $(\Omega, \F, (\F_t)_{t \ge0}, P)$
supports a Brownian motion $W = (W_t)_{t \ge0}$.
\end{ass}}%

An assumption that the underlying probability space is sufficiently
large, such as Assumption~\ref{assB}, is clearly necessary. For
example, if the sigma algebra $\F$ is of finite cardinality, then any
finitely additive probability measure is automatically countably
additive. So if $P$ charges every nonempty element of $\F$, then one
cannot have a F\"ollmer finitely additive measure for a
$P$-supermartingale $Z$ that is not a $P$-martingale.

%
\begin{teo} \label{Tapprox}
Under Assumption~\ref{assB}, there exists a F\"ollmer finitely
additive measure for the $P$-supermartingale $Z$. The F\"ollmer
finitely additive measure is never unique.
\end{teo}

Section~\ref{sfinitelyadditive} contains the proof of Theorem~\ref
{Tapprox}. Note that a similar proof also yields that, under
Assumption~\ref{assB}, the $P$-supermartingale $Z$ can be
approximated, in the sense of Fatou convergence, by a sequence of
uniformly integrable nonnegative martingales; see Theorem~\ref{TFatou}
in the Appendix~\ref{AFatou}.

Observe that the stopping times in Definition~\ref{DFAFM} were assumed
to be finite. We might also consider the extended $P$-supermartingale
$\overline{Z} = (\overline{Z}_t)_{t \in\infty}$ with $\overline
{Z}_t = Z_t$ for all $t \geq0$ and with an $\F_\infty$-measurable
$\overline{Z}_\infty\in[0, \lim_{t \uparrow\infty} Z_t]$; note
that the limit exists by the supermartingale convergence theorem. This
observation then motivates the following definition.

%
\begin{defn} \label{DEFAFM}
A weakly absolutely continuous, finitely additive probability measure
$Q\in\ba_1(\Omega, \F,P)$, such that
%
\begin{eqnarray}\label{eqQSing2}
(Q\mid _{\F_\rho} )^r[A] = \E_P[
\overline{Z}_\rho\1 _A]\qquad\mbox{for all }A \in
\F_\rho
\nonumber\\[-8pt]\\[-8pt]
\eqntext{\mbox{and (possibly infinitely-valued) stopping times }\rho,}
\end{eqnarray}
is called \emph{extended F\"ollmer finitely additive measure} for
$\overline{Z}$.
\end{defn}

We obtain a similar statement as in Theorem~\ref{Tapprox}; again, the
proof of the following theorem is provided in Section~\ref{sfinitelyadditive}.

\begin{teo} \label{Tapprox2}
Under Assumption~\ref{assB}, there exists an extended F\"ollmer
finitely additive measure for the extended $P$-supermartingale
$\overline{Z}$. The extended F\"ollmer finitely additive measure is
not unique if $\E_P[\overline{Z}_\infty] < 1$. The extended F\"
ollmer finitely additive measure is unique if $\E_P[\overline
{Z}_\infty] = 1$ and $\F= \F_\infty$.
\end{teo}

Note that an extended F\"ollmer finitely additive measure for the
extended \mbox{$P$-}supermartingale $\overline{Z}$ is automatically a F\"
ollmer finitely additive measure for the $P$-supermartingale $Z$.
As a corollary of the existence result in Theorem~\ref{Tapprox2}, we
make the following observation.

\begin{cor} \label{C39}
Under Assumption~\ref{assB}, there exists a F\"ollmer purely
finitely additive measure $Q$ for the $P$-supermartingale $Z$; to wit,
there exists $Q \in\ba_1(\Omega, \F, P)$, such that (\ref
{eqQSing}) holds and such that $Q^r = 0$.
\end{cor}

\begin{pf}
Define the extended $P$-supermartingale $\overline{Z}$ as above, now
with \mbox{$\overline{Z}_\infty= 0$}.
The existence result in Theorem~\ref{Tapprox2} then yields an extended
F\"ollmer finitely additive measure $Q$ for the extended
$P$-supermartingale $\overline{Z}$. The statement now follows from
the simple observation that $\dd Q^r / \dd P \leq\break  \dd(Q\mid _{\F_\infty
})^r / \dd P = 0$ from~(\ref{eqQSing2}) with~$\rho= \infty$.
\end{pf}
Note that Corollary~\ref{C39} includes the case that $Z$ is a
uniformly integrable \mbox{$P$-}local martingale; for example, consider $Z_t
= 1$ for all $t \geq0$. As a consequence, Corollary~\ref{C39}
illustrates that a sequence of probability measures, given on the sigma
algebras $\F_t$ for all $t \geq0$, cannot uniquely be extended to the
sigma algebra $\F_\infty$ within the class of finitely additive
probability measures; however, uniqueness holds within the class of
countably additive probability measures due to a \mbox{pi-lambda} argument. We
elaborate further on this point by discussing the case of the Lebesgue
measure on $[0,1]$ in Appendix~\ref{AE}.

\section{Proofs: F\"ollmer countably additive measure on the path space}\label{sec4}\label{sfoellmeronpathspace}

This section contains the proofs of the existence and uniqueness
results for the countably additive case in Section~\ref{Sresults}.
Before providing the proofs, we discuss some motivating examples to
outline the construction of the F\"ollmer countably additive measure.
Then we first give the proof of existence, and afterward the proof of
the assertions concerning uniqueness.

\subsection{Motivating examples}\label{sec4.1} \label{SSmotivating}
We start by discussing two illustrative examples.

%
\begin{ex}\label{exreversefoellmer}
Let $Q$ be a probability measure on the sigma algebra $\F$ and let $Y
= (Y_t)_{t \geq0}$ be a uniformly integrable nonnegative
$Q$-martingale that starts in $1$ and jumps to $0$ with positive
probability; that is,
assume that $Q[\tau<\infty, Y_{\tau-} \neq0]>0$, where $\tau=\inf
\{t \ge0\dvtx  Y_t = 0\}$. Next, define the probability measure $P$ by
$P(\dd\omega) = Y_\infty(\omega) Q(\dd\omega)$.
Then the process $Z = (Z_t)_{t \geq0}$ with $Z_t = \1_{\{t < \tau\}
}/Y_t$ is a strictly positive $P$-supermartingale, but it is not a
$P$-local martingale: fix $s<t$ and $A \in\F_s$. Then the inequalities
\[
\E_P[\1_A Z_t] = \E_Q \biggl[
\1_A \frac{1}{Y_t} \1_{\{t < \tau\}} Y_t \biggr]
\le\E_Q \biggl[ \1_A \frac{1}{Y_s}
\1_{\{s < \tau\}} Y_s \biggr] = \E_P[\1_A
Z_s]
\]
show that $Z$ is a $P$-supermartingale. Now let $(\tau_n)_{n \in\N
}$ be a nondecreasing sequence of stopping times such that $1= \E
_P[Z_{\tau_n}] = Q[\tau_n<\tau]$ holds for each $n \in\N$. Let us
show that $P[\lim_{n \uparrow\infty} \tau_n <\infty] > 0$, which
then implies that the $P$-supermartingale $Z$ is not a $P$-local
martingale. Toward this end, let $C>0$ be such that $Q[\tau\le C,
Y_{\tau-} \neq0]>0$. Observe that
\begin{eqnarray*}
P \Bigl[\lim_{n \uparrow\infty} \tau_n \le C \Bigr] &=& \lim
_{n
\uparrow\infty} P[\tau_n \le C] = \lim_{n \uparrow\infty}
\E _Q[Y_{\tau_n} \1_{\{\tau_n \le C\}}]
\\
&\ge&\E_Q \Bigl[ \Bigl(\inf_{t<\tau} Y_t
\Bigr) \1_{\{\tau\le C,
Y_{\tau-} \neq0\}} \Bigr]>0,
\end{eqnarray*}
where we used that $(\inf_{t<\tau} Y_t) >0$ on the event $\{Y_{\tau
-} \neq0\}$, which holds because $Y$ is a nonnegative $Q$-supermartingale.

The $P$-supermartingale $Z$ fails to be a $P$-local martingale
exactly because the $Q$-martingale $Y$ jumps to zero with positive
probability under the probability measure $Q$. If the $Q$-martingale
$Y$ did not jump to zero, it would be possible to stop $Y$ upon
crossing the level $1/n$ for each $n \in\N$, and this approach would
provide us with a localizing sequence of stopping times for the process
$Z$ under the probability measure $P$.
Note that, despite $Z$ not being a local martingale, it is of course
possible to construct its F\"ollmer countably additive measure on the
original space $(\Omega, \F)$: the pair $(Q,\tau)$ satisfies the
conditions in (\ref{ekunita-yoeurp}).
\end{ex}

%
\begin{ex}
Let $Z = (Z_t)_{t \geq0}$ be the $P$-supermartingale defined by $Z_t
=e^{-t}$ for all $t \geq0$. We want to interpret $Z$ as $1/Y$, where
$Y$ is a martingale under the F\"ollmer countably additive measure $Q$,
exactly as in Example~\ref{exreversefoellmer}. Since $Z$ is not a
local martingale, $Y$ must jump to zero with positive probability under
$Q$. Furthermore, $Q$ must be equivalent to $P$ before $Y$ hits zero.
This indicates that $Y_t = e^t \1_{\{t < \tau\}}$ under $Q$, where
$\tau$ is the stopping time when $Y$ hits zero.

Note that $Y$ is a martingale exactly if $\tau$ is standard
exponentially distributed and $P$ needs to satisfy $P[\tau=\infty
]=1$. In general, it is not possible to find such a stopping time $\tau
$ on $(\Omega,\F)$, think, for example, of the space $\Omega= \{0\}$
consisting only of one singleton. However, let $(\overline{\Omega},
\overline{\F}, (\overline{\F}_t)_{t \geq0})$ denote an extended
filtered space with $\overline{\Omega} = \Omega\times[0,\infty]$,
$\overline{\F} = \F\otimes\B([0,\infty])$, and $\overline{\F}_t
= \F_t \otimes\B([0,t])$, where $\B$ denotes the Borel sigma
algebra, and let $\overline{\tau}(\overline{\omega})$ denote the
second component of $\overline{\omega}$ for all $\overline{\omega}
\in\overline{\Omega}$. Then we can define the probability measures
$\overline{P} = P \otimes\delta_\infty$ and $\overline{Q} = P
\otimes\mu$ on this extended space, where $\delta_\infty$ is the
Dirac measure in infinity and $\mu$ is a standard exponential distribution.
It is not hard to check that the pair $(\overline{Q}, \overline{\tau
})$ satisfies the conditions in (\ref{ekunita-yoeurp}) with $P$ being
replaced by $\overline{P}$.

Now the crucial point is that even though a general $(\Omega, \F)$
might not be large enough to support an exponential time $\tau$, the
path space of Assumption~\ref{assP} is always large enough to
support $\tau$---as long as we allow for explosions to a cemetery
state in finite time.
In general we will not need an exponential time, but a stopping time
$\tau$ with distribution $Q[\tau> t] = \E_P[Z_t]$. However, this can
be easily reduced to the exponential case (or to the case of a uniform
variable on $[0,1]$) by a time change.
\end{ex}

The insights gained from these guiding examples allow us to construct a
F\"ollmer countably additive measure on the path space $(\Omega, \F)$
itself, rather than on the extended probability space $(\Omega\times
(0,\infty], \mathscr{P})$ used in \citet{F1972}, where
$\mathscr{P}$ denotes the predictable sigma algebra.
The crucial observation is that the F\"ollmer countably additive
measure of a local martingale can be constructed on $(\Omega, \F)$
without enlarging the space, and that a supermartingale fails to be a
local martingale if and only if under its F\"ollmer countably additive
measure, its inverse jumps to zero with positive probability. Thus, if
$(\Omega, \F)$ is large enough to allow for such a jump to zero, and
if we are able to describe what happens under the F\"ollmer countably
additive measure $Q$ once $1/Z$ jumps to zero, then we should be able
to construct the probability measure $Q$ on the sigma algebra $\F$.

In order to construct such a jump to zero, we proceed in a similar
manner as in the classical construction of killed diffusions, as
presented, for example, in Chapter~5 of \citet{ItoMcKean}: we
first introduce an independent random variable that triggers exactly
when the supermartingale loses mass, and we later forget about this
independent random variable when we project the constructed solution
down to the path space.

\subsection{F\"ollmer countably additive measure: Proof of existence}\label{sec4.2}
\label{SSconstruction}
In this subsection, we provide the proof of the existence statement in
Theorem~\ref{tfoellmercountablesummary}.

Let $Z = \MD$ be the multiplicative decomposition given in
Proposition~\ref{pmultdec}. Define the stopping times $(\widehat{\tau
}_n^M)_{n \in\N}$ and
$\widehat{\tau}{}^M$ exactly as in (\ref{eqtauM}), with $Z$ replaced
by $M$, and note that the stopped process $M^{\widehat{\tau}_n^M}$ is
a uniformly integrable martingale for each $n \in\N$. In particular,
we can define a sequence of measures $(Q^{(n)})_{n \in\N}$ by setting
$Q^{(n)}(\dd\omega) = M_{\widehat{\tau}_n^M(\omega)}(\omega)
P(\dd\omega)$. It is straightforward to check that $(Q^{(n)})_{n \in
\N}$ is consistent on the filtration $(\F_{\widehat{\tau}_n^M})_{n
\in\N}$, that is, that $Q^{(n)}(A) = Q^{(m)}(A)$ for all $A \in\F
_{\widehat{\tau}_m}$ and $m,n \in\N$ with $m \leq n$. Since the set
inclusion $\F_{\widehat{\tau}_n^M-} \subset\F_{\widehat{\tau
}_n^M}$ holds, the measures $(Q^{(n)})_{n \in\N}$ are also
consistently defined on $(\F_{\widehat{\tau}_n^M-})_{n \in\N}$.

According to Lemma~\ref{lCanProps}, the filtration $(\F_{\widehat
{\tau}_n^M-})_{n \in\N}$ is a standard system, a condition that
allows to apply Parthasarathy's extension theorem, provided in
Theorem~\ref{tparthasaratyextension}, which\vspace*{1pt} then yields the existence
of a unique probability measure $Q^M$ on $\bigvee_{n \ge0} \F
_{\widehat{\tau}_n^M-} = \F_{\widehat{\tau}{}^M-}$, such that
$Q^M\mid _{\F_{\widehat{\tau}_n^M-}}=Q^{(n)}\mid _{\F_{\widehat{\tau
}_n^M-}}$ for all $n \in\N$. Note that $P[\widehat{\tau}{}^M = \infty
] = 1$ and that
\begin{eqnarray*}
Q^M\bigl[A \cap\bigl\{t<\widehat{\tau}{}^M \bigr\}\bigr] &=& \lim_{n \uparrow\infty} Q^M\bigl[A \cap\bigl\{t<\widehat{
\tau}_n^M \bigr\}\bigr] = \lim_{n \uparrow\infty}
Q^{(n)}\bigl[A \cap\bigl\{t<\widehat{\tau}_n^M
\bigr\}\bigr]
\\
&=& \lim_{n \uparrow\infty} \E_P [M_{\widehat{\tau}_n^M} \1
_{A \cap\{t<\widehat{\tau}_n^M\}} ] = \lim_{n \uparrow\infty} \E_P[M_t
\1_{A \cap\{t<\widehat{\tau
}_n^M \}}]
\\
&=& \E_P[M_t \1_A]
\end{eqnarray*}
for all $t \geq0$ and $A \in\F_t$. Proposition~\ref{PP12} now
yields that (\ref{ekunita-yoeurp}) holds with $Q, \tau$, and $Z$
replaced by $Q^M, \widehat{\tau}{}^M$ and $M$, respectively. In
particular, if $Z$ is a local martingale, that is, if $Z \equiv M$, we
are done, as we may take any extension $\widehat{Q}{}^M$ of $Q^M$ to $\F
$ by Theorem~\ref{TC2}. Note that, in this case, we have $\widehat
{\tau}{}^M = \widehat{\tau}{}^Z$, as defined in (\ref{eqtauM}).

For the general case, we will now apply the ideas developed in
Section~\ref{SSmotivating} to construct a F\"ollmer countably additive
measure for the $P$-supermartingale $Z$. Toward this end, we define
the auxiliary space $\overline{\Omega} = \Omega\times[0,1]$ and
equip it with the sigma algebra $\overline{\F} = \F\otimes\mathcal
{B}([0,1])$, where $\mathcal{B}$ denotes the Borel sets. Let
$\overline{Q} = \widehat{Q}{}^M \otimes\mu$ denote the product
measure of $\widehat{Q}{}^M$ and $\mu$, where $\mu$ is the uniform
distribution on $[0,1]$. We will define a measurable map $\theta\dvtx
\overline{\Omega} \rightarrow\Omega$ and an $(\F_t)_{t \geq
0}$-stopping time $\tau$, such that $Q = \overline{Q}\circ\theta
^{-1}$ and $\tau$ satisfy~(\ref{ekunita-yoeurp}).

Before we continue, let us select a good version of the process $D$: we
may always suppose that $D$ is right-continuous and nonincreasing for
\emph{all} $\omega\in\Omega$, see Lemma~\ref{lcompletecadlag}.
Since $D$ starts at 1 and is nonnegative, $1-D$ is the (random)
distribution function of a measure on $[0,\infty)$ that has mass less
or equal to 1. The ``quantile function'' $\mathcal{Q}\dvtx  \Omega\times
[0,1] \rightarrow[0,\infty)$ of $1-D$ is defined as
\[
\mathcal{Q}_z(\omega) = \inf\bigl\{s \ge0\dvtx  1 - D_s(
\omega) \ge z\bigr\} = \inf\bigl\{s \ge0\dvtx  D_s(\omega) \le1 - z\bigr\}
\]
for all $z \in[0,1]$ and $\omega\in\Omega$. Note that
%
\begin{eqnarray}
\label{emeasurabilityquantilefunction}
\bigl\{(\omega,z)\dvtx  \mathcal{Q}_z(\omega) > t
\bigr\} &=& \bigl\{ (\omega,z)\dvtx  D_t(\omega) > 1-z \bigr\}
\nonumber\\[-8pt]\\[-8pt]\nonumber
&=& \bigcup_{q \in\Q\cap[0,1]} \bigl\{\omega\dvtx  D_t(
\omega) > 1-q\bigr\} \times[q,1] \in\overline{\F}.
\end{eqnarray}
Next, consider the map $\theta\dvtx  \overline{\Omega} \rightarrow\Omega
$ with
%
\begin{eqnarray}
\label{eqtheta} \theta(\omega, z) (t) = \cases{ \omega(t),&\quad$t <
\mathcal{Q}_z(\omega)$,
\cr
\Delta, &\quad$t \ge\mathcal{Q}_z(
\omega)$.}
\end{eqnarray}
To see that the map $\theta$ is measurable, it suffices to note that
\begin{eqnarray*}
\bigl\{(\omega,z)\dvtx  \theta(\omega, z) (t) \in B\bigr\} &=& \bigl\{(\omega,z)\dvtx
\omega (t) \in B, \mathcal{Q}_z(\omega) > t\bigr\};
\\
\bigl\{(\omega,z)\dvtx  \theta(\omega, z) (t) \in B \cup\{\Delta\} \bigr\} &=& \bigl
\{(\omega,z)\dvtx  \omega(t) \in B \cup\{\Delta\}, \mathcal{Q}_z(\omega) >
t \bigr\}
\\
&&{}\cup\bigl\{(\omega,z)\dvtx  \mathcal {Q}_z(\omega) \leq t\bigr\}
\end{eqnarray*}
hold for each Borel subset $B \subset E $, so that (\ref
{emeasurabilityquantilefunction}) implies $\{(\omega,z)\dvtx  \theta
(\omega, z)(t) \in B\} \in\overline{\F}$.

Next, define the stopping time $\tau= \widehat{\tau}{}^M \wedge\zeta
$ and the probability measure
$Q = \overline{Q} \circ\theta^{-1}$. Note that
\begin{eqnarray*}
Q \bigl[A \cap\{\tau> t\} \bigr] &=& \overline{Q} \bigl[\bigl\{(\omega,z)\dvtx
\omega\in A, \widehat{\tau}{}^M(\omega)>t, \zeta(\omega)>t,
\mathcal{Q}_z(\omega) > t\bigr\} \bigr]
\\
&=& \int_\Omega \biggl(\1_{A\cap\{\widehat{\tau}{}^M >
t\} \cap\{\zeta> t\}}(
\omega) \int_{[0,1]}\1_{(1-D_t(\omega
),1]}(z) \mu(\dd z) \biggr)
\widehat{Q}{}^M(\dd\omega)
\\
&=& \E_{\widehat{Q}{}^M} [D_t \1_{A\cap\{\widehat
{\tau}{}^M > t\} \cap\{\zeta> t\}} ]
= \E_P [D_t M_t\1_{A
\cap\{\zeta> t\}} ]
\\
&=& \E_P [Z_t\1_A ]
\end{eqnarray*}
for all $t \geq0$, using~(\ref{ekunita-yoeurprv}) in the second to
last step and
\[
0 \leq\E_P[Z_t \1_{A \cap\{\zeta\leq t\}}] \leq
\E_P[Z_\zeta\1_{\{
\zeta\leq t\}}] \leq\E_P[Z_\zeta
\1_{\{\zeta<\infty\}}] = 0
\]
in the last step. Another application of Proposition~\ref{PP12} then
completes the proof of the existence statement in Theorem~\ref
{tfoellmercountablesummary}.

\subsection{F\"ollmer countably additive measure: Proof of (non)uniqueness}\label{sec4.3}

Here, we provide the proofs of Lemma~\ref{lminimality}, of the
uniqueness statements of Theorem~\ref{tfoellmercountablesummary} and
of Proposition~\ref{PECountable}.\vadjust{\goodbreak}

\begin{pf*}{Proof of Lemma~\ref{lminimality}}
Let $(Q, \tau)$ also satisfy (\ref{ekunita-yoeurp}).
By Theorem~\ref{tlubinextension} in conjunction with Lemma~\ref{lCanProps} there exists an extension $\overline{Q}$ of $Q$ from $\F
_{\tau-}$ to
$\F_{\overline{\tau}-}$, where $\overline{\tau} = \widehat{\tau
}{}^Z \vee\tau$.
Note that
\[
\overline{Q}\bigl[\widehat{\tau}{}^Z< \tau\bigr] = Q\bigl[\widehat{
\tau}{}^Z < \tau \bigr] = \E_P[Z_{\widehat{\tau}{}^Z}
\1_{\{\widehat{\tau}{}^Z< \infty\}}] = 0
\]
by Proposition~\ref{PP12} and that $\overline{Q}[\widehat{\tau}{}^Z_n
< \tau] = Q[\widehat{\tau}_n^Z < \tau] = \E_P[Z_{\widehat{\tau
}_n^Z}] = 1$. These computations imply that $\overline{Q}[\widehat
{\tau}{}^Z= \tau] =1$ since $\widehat{\tau}_n^Z(\omega) \uparrow
\widehat{\tau}{}^Z(\omega)$ for all $\omega\in\Omega$.
Therefore, $(\overline{Q}, \overline{\tau})$ also satisfies (\ref
{ekunita-yoeurp}). This again yields the uniqueness of the extension.
\end{pf*}

Now, let us prove the uniqueness statement in Theorem~\ref
{tfoellmercountablesummary}.

Concerning the equivalence in (I), note that $Q$ is uniquely determined
on $\F_{\tau-}$ by (\ref{ekunita-yoeurp}). The statement then
follows from the uniqueness result of Theorem~\ref{TC2}. The statement
in (II) is proven in Proposition~2 in \citet{Yoeurp1985}.

We next show that the statement in (c) implies the one in (d). Thus,
assume that (c) holds and let $(Q, \tau)$ also satisfy (\ref
{ekunita-yoeurp}). Then Lemma~\ref{lminimality} implies that also $(Q,
\widehat{\tau}{}^Z)$ satisfies (\ref{ekunita-yoeurp}). This yields
that $\widehat{Q}{}^Z$ and $Q$ agree on $\F_{\widehat{\tau}{}^Z-}$ and
we may apply the implication from (a) to (b).

For the reverse implication from (d) to (c), we assume that the state
space $E$ is uncountable. Thanks to the implication from (b) to (a), we
only need to show that if $Z$ is not a $P$-local martingale, then
there are two different F\"ollmer pairs for $Z$. Toward this end,
consider the family of stopping times $(\rho_x)_{x \in E}$, defined by
%
\begin{equation}
\label{ErhoX} \rho_x = \inf\{t \ge0\dvtx  \mbox{there exists }
\varepsilon> 0\mbox{ s.t. } \omega\mid _{[t, t+\varepsilon)} \equiv x\};
\end{equation}
here, the right-continuity of the filtration $(\F_t)_{t \geq0}$ is
used to guarantee that ``peeking into the future is allowed,'' and thus
each $\rho_x$ is indeed a stopping time.
Since the state space $E$ is uncountable, by Lemma~\ref
{lomegastaysincountablymanyx}, there must exist some $x^* \in E$ for
which $P[\rho_{x^*} < \infty] = 0$. We define $\widetilde{\theta}$
as in (\ref{eqtheta}) but with $\Delta$ replaced by $x^*$ and
construct, exactly as in the construction of the existence proof in
Section~\ref{SSconstruction}, a~F\"ollmer pair $(Q^*,\tau^*)$ with
stopping time $\tau^* = \widehat{\tau}{}^M \wedge\rho_{x^*}$, where
$\widehat{\tau}{}^M$ is as in Section~\ref{SSconstruction}. If
$(Q,\tau)$ is the F\"ollmer pair constructed in Section~\ref
{SSconstruction}, then we have $Q[\rho_{x*} < \tau_M] = 0$ but
$Q^*[\rho_{x*} < \tau_M] > 0$ and, therefore, $Q^* \neq Q$. This
completes the proof of Theorem~\ref{tfoellmercountablesummary}.

%
\begin{rmk}
The proof that (d) implies (c) in Theorem~\ref
{tfoellmercountablesummary} leads to the following observation: if the
state space $E$ is uncountable and if the $P$-local martingale $M$ in
the multiplicative decomposition of $Z$ is a true $P$-martingale then
a F\"ollmer countably additive measure $Q$ can be defined so that
$Q[\zeta= \infty] = 1$. In particular, in such a case, the state
space of Assumption~\ref{assP} would not need to be enlarged with
the cemetery state~$\Delta$.
\end{rmk}

\begin{pf*}{Proof of Proposition~\ref{PECountable}}
To show (A), assume first that $P[\zeta<\infty] = 0$ and let the two
pairs $(Q, \tau)$ and $(\widetilde{Q}, \widetilde{\tau})$ both
satisfy (\ref{ekunita-yoeurp}). We obtain \mbox{$Q[\tau\leq\zeta] = 1$}
from (\ref{ekunita-yoeurprv}) with $G \equiv1$ and $\rho= \zeta$.
Assume now that $\{\zeta\leq t\} \subsetneq\{\tau\wedge\zeta\leq
t\}$
for\vadjust{\goodbreak} some $t > 0$. Since $\{\zeta> t\}$ is an atom in $\F_t$ we then
have $\{\tau\wedge\zeta\leq t\} = \Omega$, which contradicts
$P[\tau\wedge\zeta< \infty] = 0$. Thus, we have $Q[\tau= \zeta] =
1$ and, similarly, $\widetilde{Q}[\widetilde{\tau} = \zeta] = 1$,
which, in particular, implies that $Q\mid _{\F_{\zeta-}} = \widetilde
{Q}\mid _{\F_{\zeta-}}$ and an application of Lemma~\ref{lCanProps} concludes.

Next, consider a probability measure $P$ such that $P[\zeta=1] = 1$
and the \mbox{$P$-}supermartingale $Z$ given by $Z_t = \1_{1>t}$ for all $t
\geq0$. As candidate for a stopping time, consider $\tau= 1 + \infty
\1_{\{\zeta\leq1\}}$. It is clear that $P[\tau=\infty] = 1$.
Consider next two measures $Q_1$ and $Q_2$ such that $Q_1[\zeta=2] =
1$ and $Q_2[\zeta=3] = 1$. Then $(Q_1, \tau)$ and $(Q_2, \tau)$ are
two different F\"ollmer pairs for the $P$-supermartingale $Z$.

For the claim of existence in (i), just fix $x^* \in E$ and consider a
probability measure $P$ under which $P[\rho_{x^*}= \infty] = 1$
holds, where $\rho_{x^*}$ is defined as in
(\ref{ErhoX}) and then proceed as in the proof of the implication from
(d) to (c) in Theorem~\ref{tfoellmercountablesummary}.

For the claim in (ii), assume that $E=\{0,\ldots, n\}$ or $E = \N_0$.
Let $\rho$ denote the infimum of the jump times of the canonical
process $\omega$ to another state in the state space $E$ and let $P$
denote a probability measure on the sigma algebra $\F$ so that
$P[\omega(0) = 0] = 1$, and such that at time $1$ (but not before),
the coordinate process jumps to any other state in $E\setminus\{0\}$
with strictly positive probability or stays in~$0$ with strictly
positive probability. Then we have, in particular, $P[\rho\ge1] = 1$
and $P[\rho> 1] > 0$ as well as $P[\rho= 1] > 0$. We now consider the
$P$-supermartingale~$Z$, given by $Z_t = \1_{1>t}$ for all $t \geq0$.

Assume that the pair $(Q, \tau)$ is a F\"ollmer pair. We want to show that
$Q[\tau= \zeta] = 1$ holds, which yields the uniqueness of the F\"
ollmer pair, as in the proof of (A). Toward this end, note that $Q[\tau
= 1] = 1=Q[\tau\leq\zeta]$, again, as in the proof of (A), and that
\[
\Omega= \{\rho\leq1\} \cup\{\zeta\leq1\} \cup\{\zeta\wedge\rho > 1 \}.
\]
Thus, if we can show that $Q[\zeta\wedge\rho> 1 ] = 0$ and $Q[\rho
_i = 1 ] = 0$ for all $i \in E \setminus\{0\}$, where $\rho_i$
denotes the first hitting time of level $i$ by the canonical process,
then we have $Q[\zeta\le1] = Q[\zeta\ge1] = 1$ and, therefore,
$Q[\zeta= 1 = \tau] = 1$.

Assume first that $Q[\{\tau= 1\} \cap\{\zeta\wedge\rho> 1 \}] =
Q[\zeta\wedge\rho> 1 ] >0$. Then, by Problem~1.2.2 in \citet
{KS1}, we have $\tau(\omega) =1$ for all $\omega\in\Omega$ with
$\omega(0) = 0$ and $\zeta(\omega) \wedge\rho(\omega) > 1$;
however, this would lead to the contradiction
\[
0 < P[\zeta\wedge\rho> 1] \leq P[\tau= 1] \leq P[\tau< \infty] = 0.
\]
Next, fix $i \in E \setminus\{0\}$ and assume that $Q[\rho_i = \tau]
= Q[\rho_i = 1] > 0$.
Observe that
\begin{eqnarray*}
&& \bigl\{(\omega, t) \in\Omega\times[0,\infty)\dvtx  \omega(0)=0, \tau (\omega) =
\rho(\omega) = \rho_i(\omega) = t \bigr\}
\\
&&\qquad = \bigl\{(\omega, t) \in\Omega\times B\dvtx  \omega (0)=0, \rho(
\omega) = \rho_i(\omega) = t \bigr\}
\end{eqnarray*}
for some Borel set $B \subset[0,\infty)$; see again Problem~1.2.2 in
\citet{KS1}. By assumption, we have that $1 \in B$.
This then yields that
\[
0 < P[\rho_i = 1] \leq P[\rho_i \in B] = P[\tau=
\rho_i \in B] \leq P[\tau\in B] \leq P[\tau< \infty] = 0.
\]
This contradiction completes the proof.
\end{pf*}

\section{Proofs: F\"ollmer finitely additive measure}\label{sec5}\label{sfinitelyadditive}
\subsection{F\"ollmer finitely additive measure: Proof of existence}\label{sec5.1}
We now prove the existence statement of Theorem~\ref{Tapprox} with the
help of several lemmas. The idea of the proof is a combination of
approximating the $P$-supermartingale $Z$, approximating those
approximations again, and using a compactness argument.
We split the results up in three subsections. First, we recall some
fundamental observations that will be the key component of the proof,
then we collect several useful approximations, and finally, we will put
everything together in the proof of existence of a F\"ollmer finitely
additive measure for the $P$-supermartingale $Z$.

\subsubsection{Fundamental observations}\label{sec5.1.1}

\begin{prop}[{[\citet{CSW2001}, Proposition~A.1]}] \label{PCSW}
Consider a sequence $(Q^{(n)})_{n \in\N}$ of finitely additive
probability measures in $\ba_1(\Omega, \F,P)$. Assume that $\dd
(Q^{(n)})^r/\dd P$ converges almost surely to a nonnegative random
variable $G$. Then any cluster point $Q$ of $(Q^{(n)})_{n \in\N}$ in
$L^\infty(\Omega, \F, P)^*$ satisfies $Q^r(\dd\omega) = G(\omega)
P(\dd\omega) $.
\end{prop}

This powerful result will enable us to approximate the
$P$-supermartingale $Z$, step by step, with processes for which it is
relatively simple to construct the corresponding finitely additive
probability measure. Toward this end, we shall rely on the following
consequence of the Banach--Alaoglu theorem:

%
\begin{cor}\label{cconstructionoffinitelyadditivefoellmer}
Let $(Q^{(n)})_{n \in\N}$ be a sequence of finitely additive
probability measures in $\ba_1(\Omega, \F,P)$. Assume that the
Radon--Nikodym derivatives $\dd(Q^{(n)}|_{\F_\rho})^r/\dd P|_{\F
_\rho}$ converge to $Z_\rho$ as $n$ tends to infinity almost surely,
for each finite stopping time $\rho$ (resp., to $\overline
{Z}_\rho$ for each stopping time). Then there exists a F\"ollmer
finitely additive measure for the $P$-supermartingale $Z$
(resp., an extended F\"ollmer finitely additive measure for the
extended $P$-supermartingale $\overline{Z}$).
\end{cor}
\begin{pf}
First, note that we may identify $\ba_1(\Omega, \F,P)$ with a subset
of the unit ball of $L^\infty(\Omega,\F,P)^*$. The Banach--Alaoglu
theorem then implies that $(Q^{(n)})_{n \in\N}$ has a cluster point
$Q$. Next, observe that $Q\mid _{\F_\rho}$ is also a cluster point of the
sequence $(Q^{(n)}\mid _{\F_\rho})_{n \in\N}$.
We conclude the argument with an application of
Proposition~\ref{PCSW}.
\end{pf}

To illustrate the approach we shall follow, and for later use, we now
discuss the case that $Z$ is a $P$-local martingale:

\begin{ex}\label{exlocalmartingalefinitelyadditivemeasure}
Assume that $Z$ is a $P$-local martingale. Then there exists a F\"
ollmer finitely additive measure for $Z$. To see this, let $(\rho
_n)_{n \in\N}$ denote a sequence of localizing stopping times for
$Z$. Then, for each $n \in\N$, the uniformly integrable
$P$-martingale $Z^{\rho_n}$ defines a probability measure $Q^{(n)}$
that is absolutely continuous\vadjust{\goodbreak} with respect to $P$. Since $Q^{(n)}$ has
no singular part, for each finite stopping time $\rho$ the
Radon--Nikodym derivative $\dd(Q^{(n)}\mid _{\F_\rho})^r/\dd P\mid _{\F
_\rho}$ is given by $Z^{\rho_n}_\rho$ and, therefore, converges
almost surely to $Z_\rho$ as $n$ tends to~$\infty$ (indeed the null
set outside of which convergence takes place does not depend on $\rho
$). Corollary~\ref{cconstructionoffinitelyadditivefoellmer} now
implies the existence of a F\"ollmer finitely additive measure for $Z$.
\end{ex}

\subsubsection{Approximations}\label{sec5.1.2} \label{SSapprox}
Recall that the Doob--Meyer decomposition of the \mbox{$P$-}supermartingale
$Z$ is given by $Z = M + D$, where $M$ is a $P$-local martingale and
$D$ is a predictable nonincreasing process with $D_0 = 0$. This
decomposition is unique up to indistinguishability.
Example~\ref{exlocalmartingalefinitelyadditivemeasure} indicates that
the local martingale component of $Z$ can be handled easily. Thus, in
the following, we shall focus mostly on approximating the nonincreasing
process $D$. Toward this end, we introduce the notion of a simple
process, which, in particular, has c\`adl\`ag paths.

%
\begin{defn}
A process $G= (G_t)_{t \geq0}$ is called \emph{simple process} if
there exists a strictly increasing sequence of stopping times $(\rho
_n)_{n \in\N_0}$ and a sequence of random variables $(H_n)_{n \in\N
}$ such that $\rho_0(\omega) =0$ and $\lim_{n \uparrow\infty} \rho
_n (\omega) = \infty$ for all $\omega\in\Omega$, $H_n$ is $\F
_{\rho_{n-1}}$-measurable for all $n \in\N$, $H_0$ is $\F
_0$-measurable, and
\[
G_t(\omega) = H_0(\omega) \1_{t = 0} + \sum
_{n =1}^\infty H_n(\omega )
\1_{(\rho_{n-1}(\omega),\rho_{n}(\omega)]}(t)
\]
holds for all $\omega\in\Omega$ and $t \geq0$.
\end{defn}

%
\begin{lem}\label{lapproximatedecreasingbysimpleprocesses}
Let $G = (G_t)_{t \geq0}$ be a nonincreasing adapted process with c\`
adl\`ag paths. Then there exists a sequence of nonincreasing simple\vspace*{1pt}
processes $(G^{(k)})_{k \in\N}$ with $G^{(k)} = (G_t^{(k)})_{t \geq
0}$ such that almost surely $G^{(k)}_0= G_0$, $\lim_{k \uparrow\infty
} G^{(k)}_t = G_{t-}$, and $G^{(k)}_t \ge G_t$ for all $n \in\N$ and
$t \geq0$.
\end{lem}

\begin{pf}
It suffices to set
\[
G^{(k)}_t(\omega) = G_0(\omega)
\1_{t = 0} + \sum_{n=0}^{k 2^{k}-1}
G_{n 2^{-k}}(\omega) \1_{(n2^{-k}, (n+1)2^{-k}]}(t) + G_k(\omega)
\1_{(k, \infty)}(t)
\]
for all $\omega\in\Omega$ and $t \geq0$.
\end{pf}

The crucial observation now is that every nonincreasing simple process
is the limit of a sequence of local martingales, at least as long as
the filtered probability space is rich enough to support a Brownian motion.
Before we discuss the general result, the next example illustrates that
such an approximation is possible.

\begin{ex}\label{exonejumpsupermartingaleFatoulimit}
Under Assumption~\ref{assB}, let $ Z$ be a deterministic process
with $Z_t=1$ for all $t \in[0,1)$ and $Z_t = a \in[0,1]$ for all $t
\in[1, \infty)$. Then there exists a F\"ollmer finitely additive
measure for $Z$. To see this, define the continuous\vadjust{\goodbreak} \mbox{$P$-}local
martingales $(\mathcal{E}^{(m)})_{m \in\N}$ with $\mathcal{E}^{(m)}
= (\mathcal{E}^{(m)}_t)_{t \geq0}$ and $(N^{(m)})_{m \in\N}$ with
$N^{(m)} = (N^{(m)}_t)_{t \geq0}$ by
\begin{eqnarray*}
\mathcal{E}^{(m)}_t = \cases{ 1, &\quad$t
\in[0,1-2^{-m})$,
\vspace*{2pt}\cr
\displaystyle\exp \biggl( \int_{1-2^{-m}}^t
\frac{1}{\sqrt{1-s}}\,\dd W_s - \frac
{1}{2} \int
_{1-2^{-m}}^t \frac{1}{1-s} \,\dd s \biggr), &\quad$t
\in [1-2^{-m}, 1)$,
\vspace*{2pt}\cr
0, &\quad$t \in[1,\infty)$}
\end{eqnarray*}
and
\[
N^{(m)}_t = 1 + \int_0^t
(1-a) \frac{\mathcal{E}^{(m)}_s}{\sqrt
{1-s}} \1_{[1-2^{-m},1)}(s) \,\dd W_s.
\]
Then $N^{(m)} = a + (1-a)\mathcal{E}^{(m)}$ for each $n \in\N$; in
particular $N^{(m)} _t = 1$ for all $t \in[0,1-2^{-m}]$, and $N^{(m)}
_t = a$ for
all $t \in[1, \infty)$.

Therefore, if $\rho$ is a finite stopping time, then $N^{(m)}_\rho$
converges almost surely to~$Z_{\rho}$, and the null set outside of
which the convergence holds does not depend on $\rho$. By
Corollary~\ref{cconstructionoffinitelyadditivefoellmer} in conjunction
with Example~\ref{exlocalmartingalefinitelyadditivemeasure}, there
exists a F\"ollmer finitely additive measure for $Z$.
\end{ex}

%
\begin{lem}\label{lapproximatesimplebylocalmartingales}
Under Assumption~\ref{assB}, let $G = (G_t)_{t \geq0}$ be a
nonincreasing simple process. Then there exists a sequence of
continuous local martingales $(N^{(m)})_{m \in\N}$ with
$N^{(m)} = (N_t^{(m)})_{t \geq0}$ such that almost surely $\lim_{m
\uparrow\infty} N^{(m)}_\rho= G_{\rho}$, $N^{(m)}_0 = G_0$ and
$N^{(m)}_\rho\ge G_\rho$ for all finite stopping times $\rho$ and $m
\in\N$.
\end{lem}

\begin{pf}
The sequence of local martingales $(N^{(m)})_{m \in\N}$ can be
constructed in the same manner as described in Example~\ref
{exonejumpsupermartingaleFatoulimit}. Toward this end, define again
certain stochastic exponentials as follows. Let the sequence of
stopping times $(\rho_n)_{n \in\N}$ denote the (well-ordered) jump
times of $G$, set $\rho_0 =0$, and define the continuous $P$-local
martingales $(\mathcal{E}^{(m,n)})_{m,n\in\N_0}$ with $\mathcal
{E}^{(m,n)} = (\mathcal{E}^{(m,n)}_t)_{t \geq0}$ by $\mathcal
{E}^{(m,n)}_t = 1$ for all $t \in[0,\rho_n)$,
\[
\mathcal{E}^{(m,n)}_t = \exp \biggl(\int_{\rho_{n}}^{t}
\frac
{1}{\sqrt{\rho_n+2^{-m}-s}} \,\dd W_s - \frac{1}{2} \int
_{\rho
_n}^{t} \frac{1}{\rho_n+2^{-m}-s} \,\dd s \biggr)
\]
for all $t \in[\rho_n, \rho_n+2^{-m})$, and $ \mathcal{E}^{(m,n)}_t
= 0$ for all $t \ge\rho_n+2^{-m}$.

Next, for each $m \in\N_0$, define the ``suicide strategy'' $H^{(m)}
= (H^{(m)}_t)_{t \geq0}$ by
\[
H^{(m)}_t = \sum_{n = 0}^\infty(G_{\rho_n}
- G_{\rho_{n}+}) \frac
{\mathcal{E}^{(m,n)}_t}{\sqrt{\rho_n+2^{-m}-t}} \1_{[\rho_n,\rho
_n+2^{-m})}(t)
\]
and construct the local martingale $N^{(m)} = (N^{(m)}_t)_{t \geq0}$ by
\[
N^{(m)}_t = G_0 + \int_0^t
H^{(m)}_s \,\dd W_s.
\]
Almost surely, for each $m,n\in\N_0$, on the event $\{\rho_{n+1} -
\rho_n > 2^{-m}\}$ we have $N^{(m)}_t = G_{t}$ for all $t \in[\rho
_n+2^{-m}, \rho_{n+1}]$. This implies that almost surely $\lim_{m
\uparrow\infty} N^{(m)}_\rho= G_{\rho}$ holds for all finite
stopping times $\rho$.
\end{pf}

So far, we have approximated the process $D_-$, where $D$ is the
nondecreasing process in the Doob--Meyer decomposition of the
$P$-supermartingale $Z$, by simple processes, and we approximated
those simple processes by local martingales. What remains to be shown
is how to pass from the left-continuous process $D_-$ to the
right-continuous process $D$. In the following lemma, we will provide
the key component for this step, using the fact that the process $D$ is
predictable. The proof of the next lemma is tedious but the underlying
idea for it is very simple.

To illustrate that simple idea, let $i \in\N$ be a positive constant
and let $\sigma_1$ denote the first time that $D$ jumps down by more
than $1/i$. This jump time is predictable, thus, in particular, there
exists an announcing sequence $(\sigma^{(j)}_1)_{j \in\N}$ for
$\sigma_1$.
With the help of these stopping times, we define the
$P$-supermartingale ${Z}^{(j)}$ as
\[
{Z}^{(j)}_t = \cases{ \displaystyle\E[Z_{\sigma_1}\mid
\F_t], &\quad$t \in\bigl[\sigma^{(j)}_1,
\sigma_1\bigr]$;
\vspace*{2pt}\cr
Z_t, &\quad otherwise,}
\]
for each $j \in\N$.
Then the expectation of ${Z}^{(j)}$ is constant on $[\sigma^{(j)}_1,
\sigma_1]$, and, in particular, the $P$-supermartingale ${Z}^{(j)}$
can be decomposed in a local martingale and a nonincreasing process
that stays constant on an interval before the stopping time $\sigma
_1$. One can now show, and that is what the proof of Lemma~\ref
{lleft-continuousapprox} will do, that this local martingale plus the
left-continuous version of the nonincreasing process converges to the
$P$-supermartingale~$Z$ at time $\sigma_1$ and to $M + D_-$ at all
other times, as $j$ tends to infinity.

%
\begin{lem}\label{lleft-continuousapprox}
Let $Z$ have Doob--Meyer decomposition $Z = M + D$ and fix $i \in\N$.
Then there exists a sequence of $P$-local martingales
$({M}^{(i,j)})_{j \in\N}$ with ${M}^{(i,j)} = ({M}^{(i,j)}_t)_{t \geq
0}$ and ${M}_0^{(i,j)} =1$ and a sequence of c\`adl\`ag, adapted,
nonincreasing processes $({D}^{(i,j)})_{j \in\N}$ with ${D}^{(i,j)} =
({D}^{(i,j)}_t)_{t \geq0}$ and ${D}_0^{(i,j)} =0$ such that the
$P$-supermartingales $ {M}^{(i,j)} + {D}^{(i,j)}$ are nonnegative and
such that almost surely
\[
\lim_{j \uparrow\infty} \bigl({M}^{(i,j)}_\rho+
{D}^{(i,j)}_{\rho
-} \bigr) = M_\rho+ (
\1_{\{\Delta D_\rho\leq-1/i\}} D_\rho+ \1_{\{\Delta D_\rho> -1/i \}} D_{\rho-} )
\]
for each finite stopping time $\rho$.
\end{lem}

\begin{pf}
We shall work on the completion of $(\Omega,\F, (\F_t)_{t \geq0},
P)$, so that we can assume $D$ to be predictable and c\`adl\`ag for
\emph{all} $\omega\in\Omega$. Once we constructed $M^{(i,j)}$ and
$D^{(i,j)}$ on this completion, we may switch to indistinguishable
versions that are adapted to $(\F_t)_{t \geq0}$; see Lemma~\ref{lcompletecadlag} in the Appendix~\ref{Afiltration}.
Define $\sigma_0 = 0$ and the sequence of stopping times $(\sigma
_n)_{n \in\N}$ iteratively by
\[
\sigma_n = \inf \biggl\{ t > \sigma_{n-1}\dvtx  \Delta
D_t \leq-\frac
{1}{i} \biggr\}.
\]
Since the process $D$ is predictable, the jump time $\sigma_n$ is a
predictable time for each $n \in\N$; thus, there exists
an announcing sequence $(\sigma^{(j)}_n)_{j \in\N}$, such that $\lim_{j \uparrow\infty} \sigma^{(j)}_n(\omega) = \sigma_n(\omega)$
and $\sigma^{(j)}_n(\omega) \le\sigma^{(j+1)}_n(\omega) < \sigma
_n(\omega)$ for all $j \in\N$ and $\omega\in\Omega$; see 1.2.16
in \citet{JacodS}.

Next, since $\sigma_{n-1} < \sigma_n$ holds on the event $\{\sigma
_{n-1} < \infty\}$ for all $n \in\N$, we may assume, without loss of
generality, that $\sigma^{(j)}_n \geq\sigma_{n-1}$ holds with strict
inequality on the event $\{\sigma_{n-1} < \infty\}$ for all $n,j \in
\N$; if not, we just replace $\sigma^{(j)}_n $ by
\[
\inf\bigl\{t > \sigma_{n-1} \vee\sigma^{(j-1)}_n\dvtx  t
= \sigma ^{(k)}_n\mbox{ for some } k \in\N\bigr\}.
\]
Thus, we have $\sigma^{(j)}_n \in(\sigma_{n-1},\sigma_n)$ on the
event $\{\sigma_{n-1} < \infty\}$ for all $n,j \in\N$.

For all $j \in\N$, define now the processes ${M}^{(i,j)}$ and
${D}^{(i,j)}$ by
\begin{eqnarray*}
{M}^{(i,j)}_t &=& M_t
+ \sum_{n=1}^\infty \bigl(
\1_{\{\sigma_n^{(j)} \leq t\}} \bigl( \E_P [Z_{\sigma_n}\mid
\F_t ] - \E_P [Z_{\sigma
_n}\mid \F_{\sigma_n^{(j)}} ]
- M_{\sigma_n\wedge t} + M_{\sigma_n^{(j)}} \bigr) \bigr);
\\
{D}^{(i,j)}_t &=& D_t + \sum
_{n=1}^\infty \bigl( \1_{\{\sigma_n^{(j)}
\leq t\}} \bigl(
\E_P [Z_{\sigma_n}\mid \F_{\sigma_n^{(j)}} ] - M_{\sigma_n^{(j)}}
- D_{\sigma_n\wedge t} \bigr) \bigr)
\end{eqnarray*}
for all $t \geq0$, where we always take the same version of the
conditional expectations for the two processes and a c\`adl\`ag
modification of the processes $( \E_P [Z_{\sigma_n}\mid \F_t
])_{t \geq0}$ for all $n \in\N$. Clearly, the processes
${M}^{(i,j)}$ and ${D}^{(i,j)}$ are c\`adl\`ag, satisfy ${M}^{(i,j)}_0
= 1$ and ${D}^{(i,j)}_0 = 0$, and are $P$-local martingales, or
nonincreasing, respectively, for all $j \in\N$. We compute
\[
{M}^{(i,j)}_t + {D}^{(i,j)}_t =
Z_t + \sum_{n=1}^\infty \bigl( \1
_{\{\sigma_n^{(j)} \leq t\}} \bigl( \E_P [Z_{\sigma_n}\mid \F
_t ] - Z_{\sigma_n \wedge t} \bigr) \bigr)
\]
for all $t \geq0$, which, in particular, yields that $ {M}^{(i,j)} +
{D}^{(i,j)}$ is nonnegative for each $j \in\N$.

Moreover, fix a finite stopping time $\rho$ and note that
\begin{eqnarray*}
&& {M}^{(i,j)}_\rho+ {D}^{(i,j)}_{\rho-}
\\
&&\qquad =
M_\rho+ D_{\rho-}
\\
&&\quad\qquad{} + \sum_{n=1}^\infty \bigl(
\1_{\{\sigma_n^{(j)}
\leq\rho\}} \bigl( \E_P [Z_{\sigma_n}\mid
\F_\rho ] - M_{\sigma_n \wedge\rho} - D_{\rho-} \1_{\{\rho\leq\sigma_n\}}
-D_{\sigma_n} \1_{\{\rho> \sigma_n\}} \bigr) \bigr)
\\
&&\quad\qquad{} - \sum_{n=1}^\infty \bigl(
\1_{\{\sigma_n^{(j)} =
\rho\}} \bigl( \E_P [Z_{\sigma_n}\mid
\F_{\sigma
_n^{(j)}} ] - M_{\sigma_n^j} - D_{\rho-} \bigr) \bigr)
\end{eqnarray*}
for each $j \in\N$.
For each $\omega\in\Omega$, there exists maximally finitely many $j
\in\N$ such that the identity $\sigma_n^{(j)}(\omega) = \rho
(\omega)$ holds for some $n \in\N$. Thus, a-fortiori, the last sum
converges to zero as $j$ tends to infinity. For studying the first sum, fix
$n \in\N$. Then we want to show that
%
\begin{equation}
\label{eq58new} \quad\lim_{j \uparrow\infty} \bigl(\1_{\{\sigma_n^{(j)} \leq\rho\}} \bigl(
\E_P [Z_{\sigma_n}\mid \F_\rho ] - Z_{\sigma_n
\wedge\rho}
+ \Delta D_{\rho} \1_{\{\rho\leq\sigma_n\}} \bigr) \bigr) = \1_{\{\rho= \sigma_n\}}
\Delta D_\rho
\end{equation}
almost surely, where the null set on which the equality does not hold,
can be chosen independently of the stopping time $\rho$. This then
proves the statement. Path-by-path, (\ref{eq58new}) holds on the event
$\{\rho< \sigma_n\}$; thus, we only need to argue that $\E_P
[Z_{\sigma_n}\mid \F_\rho ] = Z_{\sigma_n}$ holds on the event $\{
\rho\geq\sigma_n\}$ almost surely (independently of the choice of
$\rho$). To see this, note that almost surely $\E_P[Z_{\sigma_n} \mid \F
_q] \mathbf{1}_{\sigma_n\le q} = Z_{\sigma_n} \mathbf{1}_{\sigma
_n\le q}$ holds for all $q \in\Q\cap[0,\infty)$ and recall that we
chose a right-continuous modification of $(\E_P[Z_{\sigma_n} \mid \F
_t])_{t \ge0}$.
\end{pf}

The next result summarizes the approximation results that we obtained
so far in this subsection.

\begin{prop}\label{PapproxofZ}
Under Assumption~\ref{assB}, there exists a family\break
$(L^{(i,j,k,m,n)})_{i,j,k,m,n \in\N}$ of\vspace*{1pt} uniformly integrable
nonnegative $P$-martingales with $L^{(i,j,k,m,n)} =
(L^{(i,j,k,m,n)}_t)_{t \geq0}$ and $\E_P[L^{(i,j,k,m,n)}_0]=1$ for
all $i,j,k,m,n \in\N$, such that almost surely
%
\begin{equation}
\label{eq59convergence} \lim_{i \uparrow\infty} \lim_{j\uparrow\infty} \lim
_{k \uparrow
\infty} \lim_{m \uparrow\infty}\lim_{n \uparrow\infty}
L^{(i,j,k,m,n)}_\rho= Z_\rho
\end{equation}
for all finite stopping times $\rho$.
\end{prop}

\begin{pf}
The statement follows by first applying Lemma~\ref
{lleft-continuousapprox}, then approximating the corresponding
processes ${D}^{(i,j)}$ via Lemmas~\ref{lapproximatedecreasingbysimpleprocesses} and \ref
{lapproximatesimplebylocalmartingales} for each $i,j \in\N$, and
finally localizing the approximating local martingales.
\end{pf}

We emphasize that there exists one $P$-null set outside of which (\ref
{eq59convergence}) holds for all stopping times $\rho$.
Theorem~\ref{TFatou} in the Appendix~\ref{AFatou} provides a similar statement as
Proposition~\ref{PapproxofZ}, but with one limit (instead of five
limits) only, thus yielding the existence of a sequence of uniformly
integrable nonnegative martingales Fatou converging to the
$P$-supermartingale $Z$.

Now observe that if there exists a deterministic time $T>0$ such that
$Z_{T+t} = Z_T$ for all $t \ge0$, and if we set $Z_\infty= \lim_{t
\uparrow\infty} Z_t = Z_T$, then by construction of the martingales
$L^{(i,j,k,m,n)}$, the convergence in (\ref{eq59convergence}) extends
to general stopping times, not necessarily finite. This allows us to
approximate extended $P$-supermartingales by an additional limit procedure.

%
\begin{cor}\label{capproxofextendedZ}
Under Assumption~\ref{assB}, let $(\overline{Z}_t)_{t \in
[0,\infty]}$ be an extended nonnegative $P$-supermartingale with $\E
_P[\overline{Z}_0] = 1$. There exists a family\break  $(\overline
{L}{}^{(h,i,j,k,m,n)})_{h,i,j,k,m,n \in\N}$ of uniformly integrable
nonnegative $P$-martingales with $\overline{L}{}^{(h,i,j,k,m,n)} =
(\overline{L}{}^{(h,i,j,k,m,n)}_t)_{t \in[0,\infty]}$ and $\E
_P[\overline{L}{}^{(h,i,j,k,m,n)}_0]=1$ for all\break  $h,i,j,k,m,n \in\N$,
such that almost surely
%
\begin{equation}
\label{eqextendedsupermartingaleapprox} \lim_{h\uparrow\infty} \lim_{i \uparrow\infty} \lim
_{j\uparrow
\infty} \lim_{k \uparrow\infty} \lim_{m \uparrow\infty}
\lim_{n
\uparrow\infty} \overline{L}{}^{(h,i,j,k,m,n)}_\rho=
\overline {Z}_\rho
\end{equation}
for all stopping times $\rho$.
\end{cor}

\begin{pf*}{Proof of Corollary~\ref{capproxofextendedZ}}
For fixed $h \in\N$, consider the family\break  $(\widetilde
{L}{}^{(h,i,j,k,m,n)})_{i,j,k,m,n \in\N}$ of uniformly integrable
nonnegative $P$-martingales with $\E[\widetilde
{L}{}^{(h,i,j,k,m,n)}_0] = 1$, which is given by Proposition~\ref
{PapproxofZ} and which approximates the $P$-super\-martingale
\[
\widetilde{Z}{}^{(h)}_t = \1_{t < h}
\frac{Z_t - \E_P[\overline
{Z}_\infty\mid \F_t]}{\E_P[Z_0 - \overline{Z}_\infty]}
\]
for all $t \geq0$, where we set $0/0=1$. As remarked above, the
convergence in (\ref{eq59convergence}) extends to general stopping
times if we set $\widetilde{Z}{}^{(h)}_\infty= \widetilde{Z}{}^{(h)}_h =
0$. Since $\E_P[\overline{Z}_\infty\mid \F_\infty] = \overline
{Z}_\infty$, it now suffices to set
\[
\overline{L}{}^{(h,i,j,k,m,n)}_t = \E_P[
\overline{Z}_\infty\mid \F_t] + \E_P[Z_0
- \overline{Z}_\infty] \widetilde{L}{}^{(h,i,j,k,m,n)}_t
\]
for all $t \in[0,\infty]$ and all $h,i,j,k,m,n \in\N$.
\end{pf*}

\subsubsection{Proofs of Theorems~\texorpdfstring{\protect\ref{Tapprox}}{3.7} and \texorpdfstring{\protect\ref{Tapprox2}}{3.9}, existence}\label{sec5.1.3}

With the help of the auxiliary results of the last two subsections, the
construction of F\"ollmer finitely additive measures is now simple. We
start by using the Radon--Nikodym derivatives
$(L^{(i,j,k,m,n)})_{i,j,k,m,n \in\N}$ from Proposition~\ref
{PapproxofZ} to construct a family of probability measures. 
Applying Corollary~\ref{cconstructionoffinitelyadditivefoellmer} then
five times yields the existence of a F\"ollmer finitely additive
measure for the $P$-super\-martingale $Z$.

In the same manner, Corollary~\ref{capproxofextendedZ} implies the
existence of an extended F\"ollmer finitely additive measure for the
extended $P$-supermartingale $\overline{Z}$.

\subsection{F\"ollmer finitely additive measure: Proof of nonuniqueness}\label{sec5.2} \label{SSnonuniqueness}

Before we get to the question of uniqueness, let us first illustrate
that a F\"ollmer finitely additive measure needs to satisfy (\ref
{eqQSing}) for all stopping times, it does not suffice to verify (\ref
{eqQSing}) only for deterministic times.

%
\begin{ex}\label
{excounterexampledeterministictimesimplystoppingtimesfinitelyadditive}
Let $(\Omega,\F,P)$ be a probability space that supports a Brownian
motion $W = (W_t)_{t \geq0}$ and an independent random variable $\rho
$ with uniform distribution on $[1,2]$. Define a filtration $(\F_t)_{t
\geq0}$ by $\F_t = \bigcap_{s>t} (\sigma(W_r\dvtx\break   r \le s) \vee\sigma
(\rho))$ for all $t \ge0$. Since $\rho$ and $W$ are independent, $W$
is a Brownian motion in the filtration $(\F_t)_{t \ge0}$. Moreover,
$\rho$ is $\F_0$-measurable and therefore a stopping time. Define
now the $P$-supermartingale $Z = (Z_t)_{t \geq0}$ by $Z_t = \1
_{[0,\rho)}(t)$ for all $t \ge0$.

In the same way as in Example~\ref
{exonejumpsupermartingaleFatoulimit}, we can now construct two
sequences $(M^{(m)})_{m \in\N}$ and $(N^{(m)})_{m \in\N}$ of
continuous, nonnegative local martingales with $M^{(m)} =
(M^{(m)}_t)_{t \geq0}$ and $N^{(m)} = (N^{(m)}_t)_{t \geq0}$ for all
$m \in\N$. Toward this end, note that $\rho-1/m$ is a stopping time
for each $m \in\N$. Then, for each $m \in\N$, let $M^{(m)}$ be a
local martingale that is constant $1$ until time $\rho-1/m$,
fluctuates in the interval $[\rho-1/m,\rho]$, and is constant $0$
after time $\rho$, and let
$N^{(m)}$ be a local martingale that is constant $1$ until time $\rho
$, fluctuates in the interval $[\rho,\rho+1/m]$, and is constant $0$
after time $\rho+1/m$.

Since $\rho$ is absolutely continuous, we have almost surely $\lim_{m
\uparrow\infty} M^{(m)}_t = Z_t = \lim_{m \uparrow\infty}
N^{(m)}_t$ for all $t \ge0$. So if $Q_1$ is a cluster point of the F\"
ollmer finitely additive measures for $(M^{(m)})_{m \in\N}$, and if\vspace*{1pt}
$Q_2$ is a cluster point for the F\"ollmer finitely additive measures
for $(N^{(m)})_{m \in\N}$, then we have  %
\[
(Q_1\mid _{\F_t})^r[A] = \E_P[
\1_A Z_t] = (Q_2\mid _{\F_t})^r[A]
\]
for all $A \in\F_t$ and $t \geq0$. However, we have $\lim_{m
\uparrow\infty} M^{(m)}_\rho= 0 \neq1 =\break  \lim_{m \uparrow\infty}
N^{(m)}_\rho$ and, therefore, $(Q_1\mid _{\F_\rho})^r \neq(Q_2\mid _{\F
_\rho})^r$.
\end{ex}

In order to prove the (non)uniqueness results of Theorems~\ref
{Tapprox} and \ref{Tapprox2}, we first prove some important special
cases in auxiliary lemmas. We shall use the convention $Z_\infty= \lim_{t \uparrow\infty} Z_t$, but warn the reader that it is possible
that $\overline{Z}_\infty\neq Z_\infty$; however, we always have
$\overline{Z}_\infty\in[0,Z_\infty]$.

%
\begin{lem} \label{1111}
Under Assumption~\ref{assB}, suppose that $P[Z_\infty>0]>0$.
Then there exist two F\"ollmer finitely additive measures $Q_1, Q_2$
for the $P$-super\-martin\-gale ${Z}$.
\end{lem}

\begin{pf}
Observe that under the assumption the extension $\overline{Z}$ of the
\mbox{$P$-}super\-martingale $Z$ is not unique; for example, we may set
$\overline{Z}_\infty= 0$ or $\overline{Z}_\infty= Z_\infty$.
However, for each extension there exists an extended F\"ollmer finitely
additive measure, which is also a F\"ollmer finitely additive measure
for the $P$-supermartingale $Z$. Since the measures corresponding to
different extensions of $Z$ do not agree, the statement is proven.
\end{pf}

The proof of Lemma~\ref{1111} is short but not very
insightful. We thus provide an alternative, more ``constructive'' proof
in Appendix~\ref{Aalternative} to illustrate where the lack of
uniqueness comes into play.

%
\begin{lem} \label{LnonUniqueness2}
Under Assumption~\ref{assB}, suppose that there exists $c \in
[0,1)$ for which $P[\rho< \infty] = 1$, where $\rho= \inf\{t \ge0\dvtx
Z_t \leq c\}$. Then there exist two extended F\"ollmer finitely
additive measures $Q_1, Q_2$ for the extended $P$-supermartingale
$\overline{Z}$.
\end{lem}

\begin{pf}
Recall the family $(\overline{L}{}^{(h,i,j,k,m,n)})_{h,i,j,k,m,n \in\N
}$ of uniformly integrable nonnegative martingales from Corollary~\ref
{capproxofextendedZ}.
For sake of notation, fix $h,i,j,k,m,n \in\N$ and set
$\widehat{L} = \overline{L}{}^{(h,i,j,k,m,n)}$. Define the stopping
time $\widehat{\sigma}$
by $\rho$ on the event $\{\widehat{L}_\rho> (1+c)/2\}$ and $\infty$
on its complement. Note that the convergence in (\ref
{eq59convergence}) implies that, for $P$-almost all $\omega\in\Omega
$, there exists $h^*(\omega)$, such that for all $h \ge h^*(\omega)$
there exists $i^*(\omega,h)$, such that for all $i \ge i^*(\omega,h)$
there exists $j^*(\omega,h,i), \ldots,$ such that ${\widehat{\sigma
}}(\omega) = \infty$ as long as $h \ge h^*(\omega), i \ge i^*(\omega
,h),\ldots, n \ge n^*(\omega,h,i,j,k,m)$.

Define now the events $(B_l)_{l \in\N}$ by
\[
B_{l} =\bigl\{ W_{\rho+ 1} \in (l,l+1) \bigr\}
\]
and note that $\E_P[\1_{B_l}\mid \F_{\rho}] > 0$ almost surely using the
fact that $\rho<\infty$.
For each $l \in\N$, consider the right-continuous, uniformly
integrable $P$-martingale $\widehat{L}{}^{(l)}$ with $\widehat
{L}{}^{(l)} = (\widehat{L}{}^{(l)})_{t \geq0}$, defined by
\[
\widehat{L}{}^{(l)}_\infty= \widehat{L}_{\widehat{\sigma}} \biggl(\1
_{\{\widehat{\sigma}= \infty\}} + \1_{\{\widehat{\sigma} = \rho\}
} \frac{\1_{B_l}}{\E_P [\1_{B_l}\mid \F_{\rho} ]} \biggr).
\]
Note that, for each fixed $l \in\N$, (\ref
{eqextendedsupermartingaleapprox}) holds with $\widehat{L}$ replaced
by $\widehat{L}{}^{(l)}$, for each $h,i,j,k,m,n \in\N$.
Thus, for each $l \in\N$, we obtain an extended F\"ollmer finitely
additive measure $Q^{(l)}$ for the extended $P$-supermartingale
$\overline{Z}$; see the proof of the existence statement of
Theorem~\ref{Tapprox2}.

Now, for each $l \in\N$, we have
\begin{eqnarray*}
\E_P \bigl[\widehat{L}{}^{(l)}_\infty
\1_{B_l} \bigr] &\geq& \E_P \biggl[\widehat{L}_{\widehat{\sigma}}
\1_{\{\widehat{\sigma} = \rho\}} \frac{\1_{B_l}}{\E_P [\1_{B_l}\mid \F_{\rho} ]} \biggr] = \E _P [
\widehat{L}_{\widehat{\sigma}} \1_{\{\widehat{\sigma} =
\rho\}} ]
\\
&=& 1 - \E_P [\widehat{L}_{\rho} \1_{\{\widehat{\sigma} > \rho
\}} ] \geq1
- \frac{1+c}{2} = \frac{1-c}{2} > 0;
\end{eqnarray*}
thus, we also have $Q^{(l)}[B_l] \geq(1-c)/2$.
Since the events $(B_l)_{l \in\N}$ are disjoint, there must exist
more than one extended F\"ollmer finitely additive measure for the
extended $P$-supermartingale $\overline{Z}$.
\end{pf}

The previous two lemmas now yield the proof of the nonuniqueness
assertion of Theorem~\ref{Tapprox}.

First, note that the $P$-supermartingale $Z$ always satisfies one (or
both) of the following two conditions:
\begin{longlist}[(A)]
\item[(A)] $P[Z_\infty> 0] > 0$;
\item[(B)] $P[\rho< \infty] = 1$, where $\rho= \inf\{t \ge0\dvtx  Z_t
\leq1/2\}$.
\end{longlist}
That is, either the $P$-supermartingale $Z$ has positive probability
to have a positive limit or it crosses $1/2$ almost surely in finite
time, or both. Then Lemmas~\ref{1111} and~\ref{LnonUniqueness2} yield the nonuniqueness, in both of those cases.

The corresponding statement of Theorem~\ref{Tapprox2} needs one more
lemma.%

\begin{lem} \label{LnonUniqueness3}
Under Assumption~\ref{assB}, there exist two extended F\"ollmer
finitely additive measures $Q_1, Q_2$ for the extended
$P$-supermartingale $\overline{Z}=(\1_{[0,\infty)(t)})_{t \in
[0,\infty]}$ such that
$Q_1 \neq Q_2$.
\end{lem}

\begin{pf}
Consider the $P$-local martingale $G=(G_t)_{t \geq0}$, defined by
$G_t = \int_0^t \exp(-s) \,\dd W_s$, and note that if we set $G_\infty
= \lim_{t \uparrow\infty} G_t$, then for every $t \ge0$ the random
variable $G_\infty- G_t$ is normally distributed with nontrivial
variance, and independent of $\F_t$. In particular, the two disjoint
events $A^{(+)} = \{G_\infty>0\}$ and $A^{(-)}= \{G_\infty<0\}$
satisfy $\E_P[A^{(\mathcal{y})}\mid \F_n] > 0$ almost surely for
$\mathcal{y} \in\{-,+\}$ and for all $n \in\N$.

We now construct, ``by hand,'' two sequences $(L^{(+,n)})_{n \in\N},
(L^{(-,n)})_{n \in\N}$ of nonnegative uniformly integrable
martingales with $L^{(\mathcal{y},n)} = (L^{(\mathcal{y},n)}_t)_{t
\geq0}$ such that $\lim_{n \uparrow\infty} L^{(\mathcal
{y},n)}_\rho= \1_{\{\rho<\infty\}}$ and that $\E_P[L^{(\mathcal
{y},n)}_\infty\1_{A^{(\mathcal{y})}}] = 1$ for $\mathcal{y} \in\{
-,+\}$ and\vspace*{1pt} for all $n \in\N$. This then shows the statement, by the
same arguments in the proof of the existence statement of Theorem~\ref
{Tapprox}.

To construct such sequences of nonnegative uniformly integrable
martingales, fix $\mathcal{y} \in\{-,+\}$ and let $\mathcal{E}^{(n)}
= (\mathcal{E}^{(n)}_t)_{t \geq0}$ be a continuous nonnegative local
martingale that stays constant at one up to time $n-1$ and is zero
almost surely at time $n$ for each $n \in\N$; for instance, such a
local martingale can easily be obtained by modifying the process given
in Example~\ref{exonejumpsupermartingaleFatoulimit}. Now, let $\rho
_n$ denote the first hitting time of $2^n$ by $\mathcal{E}^{(n)} $ for
each $n$. The Borel--Cantelli lemma yields that $\sum_{n = 1}^\infty\1
_{\{\rho_n<\infty\}} < \infty$ almost surely. Now, for each $n \in
\N$ define the random variable
\[
L^{(\mathcal{y},n)}_\infty= \mathcal{E}^{(n)}_{\rho_n \wedge_n}
\frac{\1_{A^{(\mathcal{y})}}}{\E_P[A^{(\mathcal{y})}\mid \F_n] },
\]
note\vspace*{1pt} that $\E_P[L^{(\mathcal{y},n)}_\infty] = 1$, and define the
uniformly integrable martingale $L^{(\mathcal{y},n)}$ as the
right-continuous modification of the process $(\E_P[L_\infty
^{(\mathcal{y},n)}\mid \F_t])_{t \geq0}$. It is simple to see that both
sequences $(L^{(+,n)})_{n \in\N}, (L^{(-,n)})_{n \in\N}$ of
nonnegative uniformly integrable martingales, constructed in this way,
satisfy the claimed conditions, which completes the proof.
\end{pf}

We are now ready to prove the uniqueness claims of Theorem~\ref{Tapprox2}.

It is clear that the extended F\"ollmer finitely additive measure is
unique whenever the $P$-supermartingale $Z$ is a uniformly integrable
martingale, $\E_P[\overline{Z}_\infty] = 1$, and $\F=\F_\infty$.
Thus, let us now assume that $\E_P[\overline{Z}_\infty] < 1$.

To make headway, write $\overline{Z}$ as the sum of a uniformly
integrable $P$-martingale $N=(N_t)_{t \in[0,\infty]}$ and an
extended $P$-supermartingale $G = \break (G_t)_{t \in[0,\infty
]}$; here $N$ is just the right-continuous modification of the\vspace*{1pt} process
$(\E_{P} [\overline{Z}_\infty\mid \F_t])_{t \geq0}$. Next, note that
there exist two extended F\"ollmer finitely additive measures $Q^N$ and
$Q^G$ corresponding to the two $P$-super\-martingales $N/\E_P[N_0]$
(if $\E_P[N_0]>0$, otherwise just use the null measure) and $G/\E_P[G_0]$.
An extended F\"ollmer finitely additive measure $Q$ for the extended
$P$-super\-martingale $\overline{Z}$ can then be constructed by
setting $Q = \E_P[N_0] Q^N + \E_P[G_0] Q^G$.
Thus, to show nonuniqueness of the extended F\"ollmer finitely additive
measure $Q$ for $\overline{Z}$, it is sufficient to show nonuniqueness
of the extended F\"ollmer finitely additive measure $Q^G$ for the
extended \mbox{$P$-}supermartingale $G =\break  (G_t)_{t \in[0,\infty]}$. For sake\vspace*{1pt}
of notation, we thus assume from now on that $\overline{Z} \equiv G$;
that is, that $\overline{Z}_\infty= 0$.

We now first consider the case that $Z$ is a $P$-uniformly integrable
martingale. Then there exists a (countably additive) probability
measure $P'$, defined by $P'(\dd\omega) = Z_\infty(\omega) P(\dd
\omega)$. For the extended $P'$-super\-martingale $(1_{[0,\infty
)}(t))_{t \in[0,\infty]}$ there exist two different extended F\"
ollmer finitely additive measures $Q_1$ and $Q_2$ according to
Lemma~\ref{LnonUniqueness3}. However, note that $Q_1$ and $Q_2$ are
also extended F\"ollmer finitely additive measures for the extended
$P$-super\-martingale $\overline{Z}$.

We next consider the case that $Z$ is not a $P$-uniformly integrable
martingale. In particular, the extended $P$-supermartingale $\overline
{Z}$ can be written as a sum of a uniformly integrable martingale and a
nonzero potential. With the same arguments as above, in order to show
nonuniqueness, we may assume, without loss of generality, that
$\overline{Z}$ is a potential; that is, in particular, that $Z_\infty
= 0 = \overline{Z}_\infty$. However, then Lemma~\ref
{LnonUniqueness2} yields the nonuniqueness of the extended F\"ollmer
finitely additive measure and the proof is complete.

\begin{appendix}
\section{Incomplete filtrations}\label{sec6}\label{Afiltration}
To dispel possible concerns about the fact that we are working with
incomplete filtrations, here we collect some observations which allow
us to transfer results from complete filtrations to our setting. Note
that there are at least two important monographs which avoid the use of
complete filtrations as far as possible, \citet{jacodbook}
and~\citet{JacodS}.

Let $(\Omega, \F, (\F_t)_{t \ge0}, P)$ be a filtered probability
space with a right-continuous filtration $(\F_t)_{t \geq0}$. Write
$\F^P$ for the $P$-completion of $\F$, and $\mathcal{N}^P$ for the
$P$-null sets of $\F^P$. Then the filtration $(\F_t^P)_{t \geq0} =
(\F_t \vee\mathcal{N}^P)_{t \geq0}$ satisfies the usual conditions.
For every random variable $X$ on $(\Omega, \F^P)$, there exists a
random variable $Y$ on $(\Omega, \F)$ with $P[X=Y]=1$.

The first result relates stopping times under $(\F_t)_{t \geq0}$ and
under $(\F^P_t)_{t \geq0}$.

%
\begin{lem}[{[\citet{JacodS}, Lemma~I.1.19]}]\label{lcompletestoppingtime}
Any stopping time on $(\Omega, \F^P, (\F^P_t)_{t \geq0})$ is almost
surely equal to a stopping time on $(\Omega, \F, (\F_t)_{t \geq0})$.
\end{lem}

Comparable results hold on the level of processes.

%
\begin{lem}\label{lcompletepredictableoptional}
Any predictable (resp., optional) process on the completion
$(\Omega, \F^P, (\F^P_t)_{t \geq0})$ is indistinguishable from a
predictable (resp., optional) process on $(\Omega, \F, (\F
_t)_{t \geq0})$.
\end{lem}

\begin{pf}
The predictable case is Lemma~I.2.17 of \citet{JacodS}. The
optional case is shown in the same way.
\end{pf}

%
\begin{lem}\label{lcompletecadlag}
Let $G = (G_t)_{t \geq0}$ be an $(\F^P_t)_{t \geq0}$-adapted
process that it almost surely c\`adl\`ag. Then $G$ is indistinguishable
from an $(\F_t)_{t \geq0}$-adapted process $\widetilde G=
(\widetilde{G}_t)_{t \geq0}$ which is right-continuous for \emph
{all} $\omega\in\Omega$ and which possesses left limits everywhere\vspace*{1pt}
except at a stopping time $\tau$ with $P[\tau=\infty] = 1$. If $G$
is almost surely nondecreasing and bounded from above, then $\widetilde
G$ can be chosen nondecreasing and bounded from above for all $\omega
\in\Omega$.
\end{lem}

\begin{pf}
For each $q \in\Q\cap[0,\infty)$, consider an $\F_q$-measurable
random variable $\overline G_q$ with $P[\overline G_q = G_q] = 1$. Then
there exists a nondecreasing sequence $(\mathcal N_t)_{t \geq0}$ of
null sets with $\mathcal N_t \in\F_t$ such that the process
$(\overline G_q)_{q \in\Q\cap[0,t]}$ has left and right limits for
all $\omega\in\Omega\setminus\mathcal N_t$; see, for example,
page~59 in~\citet{EKMarkovian} for details. We then define the
stopping time $\tau(\omega) = \inf\{t \ge0\dvtx  \omega\in\mathcal
N_t\}$ and take $\widetilde G$ as the right limit process of
$(\overline G^\tau_q)_{q \in\Q\cap[0,\infty)}$. If $G$ is almost
surely nondecreasing, define
\[
\mathcal M_t = \bigl\{\omega\in\Omega\dvtx  \exists 0\leq q <
q'\leq t \in \Q\mbox{ such that } \overline G_q(
\omega) > \overline G_{q'}(\omega )\bigr\} %
\]
for all $t \geq0$ and note that $\mathcal M_t \in\F_t$ is a
$P$-null set.
We\vspace*{1pt} now define the stopping time $\tau$ as before, but now with
$\mathcal N_t \cup\mathcal M_t$ replacing $\mathcal N_t$, for all $t
\geq0$, and set $\widetilde G_t = (G_t \wedge K) \1_{\{\tau> t\}} + K
\1_{\{\tau\leq t\}}$ for all $t \geq0$, where $K$ is an almost sure
upper bound of the process $G$.
\end{pf}

\section{Multiplicative decomposition of~a~supermartingale}\label{sec7} \label{Amultiplicative}
In this \hyperref[sec7]{Appendix}, we discuss the multiplicative decomposition of a
nonnegative supermartingale.
For the nonnegative $P$-supermartingale $Z$, we shall write
${}^{p}{Z}$ to denote its predictable projection, which is the unique
predictable process that is characterized by the identity
${}^{p}{Z}_\rho= \E_P[{}^{p}{Z}_\rho\mid  \F_{\rho-}]$ on the event $\{
\rho<\infty\}$ for all predictable stopping times $\rho$; see also
Theorem~1.2.28 in \citet{JacodS}.

Let $\rho_0$ denote the first hitting time of zero by the
$P$-supermartingale $Z$. According to Th\'eor\`eme~1 in \citet
{Jacod1978} there exists a nondecreasing sequence of finite stopping
times $(\rho_n)_{n \in\N}$ such that $(Z_{\rho_n-}) \wedge
({}^{p}{Z}_{\rho_n}) \geq1/n$ and $\lim_{n \uparrow\infty} \rho_n
= \rho_0$. Define\vspace*{1pt} the event $B=\bigcup_{n \in\N} \{\rho_n = \rho
_0\}$ and denote its complement by $B^c$.
This allows us to write $\rho_0 = \rho_0^P \wedge\rho_0^S$ for the
two stopping times $\rho_0^P = \rho_0 \1_{B^c} + \infty\1_{B^c}$ and
$\rho_0^S = \rho_0 \1_{B} + \infty\1_{B}$ It is clear that the
stopping time $\rho_0^P$ is predictable, announced by the sequence
$(\rho_n + n \1_{\{\rho_n = \rho_0\}})_{n \in\N}$.

To obtain some intuition, note that $\rho_0^P$ is finite if and only
if either one of two events occurs: either $Z$ hits zero continuously,
that is, for each $n \in\N$ it crosses the level $1/n$ strictly
before it hits zero, or $Z$ jumps to zero, but with an announced jump.
On the other side, $\rho_0^S$ is the time when $Z$ jumps to zero ``by
surprise;'' by which we mean that one has not almost sure knowledge
about that jump just before it occurs.

We are now ready to state an existence and uniqueness result for a
multiplicative decomposition of the $P$-supermartingale $Z$.

\begin{prop}[{[\citet{Yoeurp1976}, Th\'eor\`eme~3.9]}]\label{pmultdec}
There exist a nonnegative local martingale $M = (M_t)_{t \ge0}$ and a
nonnegative, nonincreasing and predictable process $D = (D_t)_{t \geq
0}$ with $D_0 = 1$ and c\`adl\`ag paths such that:
\begin{itemize}
\item$Z \equiv \MD$;
\item$M \equiv M^{\rho_0}$ and $D \equiv D^{\rho_0}$;
\item$D_{\rho_0^P} = 0$ on the event $\{\rho_0^P < \infty\}$;
\item$M$ is continuous at time $\rho_0^P$ on the event $\{\rho_0^P <
\infty\}$.
\end{itemize}
These properties determine the processes $M$ and $D$ uniquely up to
indistinguishability.
\end{prop}

In the setting of the last proposition, $D^{\rho_0}$ is a predictable
process despite the fact that $\rho_0$ is in general not a predictable
time; see Proposition~I.2.4 in~\citet{JacodS}. Thus, assuming
that $D \equiv D^{\rho_0}$ does not lead to problems.

\begin{pf*}{Proof of Proposition~\ref{pmultdec}}
The assertion follows from Th\'eor\`eme~3.9 in \citet{Yoeurp1976}
and Corollaire~8 in \citet{Jacod1978}, which yield the existence
of two nonnegative processes $N$ and $H$ with c\`adl\`ag paths, such
that $N$ is a local martingale on the stochastic interval $[0, \rho
_0^P)$ and $H$ is a nonincreasing and predictable process such that $Z
\equiv N H$ holds.
Note that we may replace the process $H$ by $D = (H_t \1_{\{\rho
_0^P>t\}})_{t \geq0}$, which is also a predictable process; see, for
example, Theorems~2.15.a and 2.28.c in \citet{JacodS}. It is easy
to check that we still have the decomposition $Z = ND$.

Next, an application of Proposition~A.4 in \citet{CFR2011} allows
us to extend $N$ to a local martingale $M$ on the whole positive half
line such that $M \equiv M^{\rho_0^P}$ and $M$ is continuous at time
$\rho_0^P$. We still have the decomposition $Z = \MD$ since $Z_{\rho
_0^P} = 0 = D_{\rho_o^P}$ on the event ${\rho_0^P<\infty}$.
Moreover, note that we may assume, without loss of generality, that $M
\equiv M^{\rho_0}$ and $D \equiv D^{\rho_0}$.

To see the asserted uniqueness, consider two processes $M'$ and $D'$ as
in the theorem. Then, Corollaire~8 in \citet{Jacod1978} yields
that $M' \equiv M$ on $[0, \rho_0^P)$. By the required continuity of
$M'$ at time $\rho_0^P$ we obtain that $M' \equiv M$. Corollaire~8 in
\citet{Jacod1978} also yields that $D' \equiv D$ on $[0, \rho
_0)$ and that $D'_{\rho_0^S} = D_{\rho_0^S}$ on the event $\{\rho
_0^S < \infty\}$. Thus, $D' \equiv D$ follows from the assumption that
$D_{\rho_0^P} = 0$ and that
$D' \equiv D'^{\rho_0}$.
\end{pf*}

We remark that \citet{Yoeurp1976} does not mention a condition
that corresponds to $D \equiv D^{\rho_0}$ in the formulation of Th\'
eor\`eme~3.9. However, simple counterexamples illustrate that such a
condition is needed to obtain uniqueness of the processes $M$ and $D$.
Both \citet{Yoeurp1976} and \citet{Jacod1978} assume that
the filtration $(\F_t)_{t \ge0}$ satisfies the usual conditions.
However, using the observations made in Appendix~\ref{Afiltration}, we
can easily dispense with that assumption.

For another multiplicative decomposition of a given nonnegative
supermartingale, see also Theorem~4.2 and Remark~4.5 in \citet
{Penner2013}. In that decomposition, however, the nonincreasing process
$D$ is not necessarily predictable.

\section{Certain spaces in measure theory}\label{sec8} \label{Ameasure}
In this \hyperref[sec8]{Appendix}, we recall some measure-theoretic concepts.

Let $(X, \X)$ and $(Y, \Y)$ be two measurable spaces. A bijection
$f\colon X \rightarrow Y$ is called \emph{isomorphism between} $(X,\X
)$ \emph{and} $(Y,\Y)$ if both $f$ and $f^{-1}$ are measurable. The spaces
$(X, \X)$ and $(Y, \Y)$ are called \emph{isomorphic} if there exists
an isomorphism between them. A bijection $\varphi\colon\X\rightarrow
\Y$ is called \emph{$\sigma$-isomorphism} if it preserves countable
set operations, that is, if $\varphi(\bigcup_{n \in\N} A_n) =
\bigcup_{n\in\N} \varphi(A_n)$ and $\varphi(\bigcap_{n \in\N}
A_n) = \bigcap_{n \in\N} \varphi(A_n)$ for each sequence $(A_n)_{n
\in\N}$ with $A_n \in\X$ for all $n \in\N$. The sigma algebras
$\X$ and $\Y$ are called \emph{$\sigma$-isomorphic} if there
exists a \mbox{$\sigma$-}isomorphism between them. Note that if the function
$f\colon X \rightarrow Y$ is an isomorphism between the measurable
spaces $(X,\X)$ and $(Y,\Y)$ then the mapping $\X\rightarrow\Y$
with $A \mapsto\{f(x)\dvtx  x \in A\}$ is a $\sigma$-isomorphism between
$\X$ and $\Y$.

A measurable space $(X,\X)$ is called \emph{countably generated} if
there exists a sequence $(A_n)_{n \in\N}$ in $\X$, such that $\X=
\sigma(A_n\dvtx  n \in\N)$. If $X$ is a separable metrizable space, then
its Borel sigma algebra $\B(X)$ is countably generated: if $B_r(x)$
denotes the open ball around $x\in X$ with radius $r \geq0$ with
respect to a metric that generates the topology, and if $(x_n)_{n \in
\N}$ is a countable dense subset, then $\{B_q(x_n)\dvtx  n \in\N, q \in
\Q, q \ge0\}$ is a countable base for the topology, and, in
particular, generates $\B(X)$.

%
\begin{defn}[{[\citet{Pa}, Definition~V.2.2]}] \label{Dstandard}
A measurable space $(X, \X)$ is called \emph{standard Borel space} if
there exists a Polish space $Y$ such that $\X$ is $\sigma
$-isomorphic to $\B(Y)$, where $\B(Y)$ denotes the Borel sigma
algebra of~$Y$.
\end{defn}

%
\begin{lem}
Any standard Borel space is countably generated.
\end{lem}

\begin{pf}
Let $(X, \X)$ denote a standard Borel space and $(Y,\B(Y))$ the
corresponding Polish space of Definition~\ref{Dstandard}. As we
remarked above, $\B(Y)$ is countably generated.
Let $(A_n)_{n \in\N}$ generate $\B(Y)$, let $\varphi\colon\X
\rightarrow\B(Y)$ denote a $\sigma$-isomorphism, and note that
\[
\X= \sigma \bigl(\bigl\{\varphi^{-1}(A_n)\dvtx  n \in\N\bigr\}
\bigr)
\]
holds, which proves the statement.
\end{pf}

%
\begin{defn} \label{DLusin}
A \emph{Lusin space} is a topological space $E$ for which there exists
a Polish space $Y$ and a continuous bijection $f\colon Y \rightarrow
E$. A \emph{state space} is a metrizable Lusin space.
\end{defn}

It can easily be seen that $E$ is a Lusin space if and only if there
exists a finer topology on $E$ under which $E$ becomes Polish. Each
Polish space is a Lusin space and a state space. An example of a state
space that is not Polish is the set $\Q\subset\R$ of rational
numbers, equipped with the Euclidean metric. For the corresponding
Polish space we may choose $Y = \Q$, equipped with the discrete topology.

We also need the notion of a standard system, introduced by F\"ollmer.
Recall that if $\X$ is a sigma algebra, then a set $A \in\X$ is
called \emph{atom} in $\X$ if $B \in\X$ and $B \subset A$ implies
$B = \varnothing$ or $B = A$.

%
\begin{defn}[{[\citet{F1972}, Appendix]}]
Let $X$ be a set and let $\mathcal{T}$ be a partially ordered nonempty
index set. Assume that $(\X_t)_{t \in\mathcal{T}}$ is an
nondecreasing sequence of sigma algebras on $X$. The ``filtration''
$(\X_t)_{t \in\mathcal{T}}$ is called \emph{standard system} if it
satisfies the following conditions:
\begin{longlist}[(ii)]
\item[(i)] The space $(X, \X_t)$ is a standard Borel space for each
$t \in\mathcal{T}$.
\item[(ii)] If $(t_n)_{n \in\N}$ is a nondecreasing sequence in
$\mathcal{T}$, and if $(A_n)_{n \in\N}$ is a nonincreasing sequence
of atoms with $A_n \in\X_{t_n}$ for each $n\in\N$, then $\bigcap_{n \in\N} A_n \neq\varnothing$.
\end{longlist}
\end{defn}

The following criterion is useful for verifying whether a given
sub-sigma algebra of a standard Borel space corresponds to a standard
Borel space.

%
\begin{lem}[{[\citet{Pa}, Theorem~V.2.4]}] \label{Lcountably}
Let $(X,\X)$ be a standard Borel space, and let $\mathcal{W} \subset
\X$ be countably generated. Then $(X, \mathcal{W})$ is a standard
Borel space.
\end{lem}

Lemma~\ref{Lcountably} yields, in particular, that each state space
$E$ is a standard Borel space: first note that $E$ is separable,
because if $Y$ is a Polish space, $(y_n)_{n \in\N}$ is dense in~$Y$,
and $f\colon Y \rightarrow E$ is a continuous bijection, then
$(f(y_n))_{n \in\N}$ is dense in~$E$. Therefore, the Borel sigma
algebra $\B(E)$ of $E$ is countably generated. If now $\widetilde{\B
}(E)$ is the Borel sigma algebra corresponding to a finer topology on
$E$ under which $E$ is Polish, then $\B(E) \subset\widetilde{\B
}(E)$, and therefore $E$ is a standard Borel space according to
Lemma~\ref{Lcountably}.

\section{Extension of measures}\label{sec9} \label{Aextension}
In this \hyperref[sec9]{Appendix}, we recall the extension theorems needed to construct
F\"ollmer countably additive measures on the path space, and to prove
their (non)uniqueness. We start with the simplest extension, when we
just want to add one set to the sigma algebra.

%
\begin{defn}
Let $(X, \mathcal{X}, \mu)$ be a probability space and let $A \subset
X$. The inner and outer content $\mu_*[A] $ and $\mu^*[A] $ of the
set $A$ are defined by
%
\begin{eqnarray}\label{eqouter}
\mu_*[A] &=& \sup\bigl\{\mu[B]\dvtx  B \in\mathcal{X}, B \subset A\bigr\} = \max\bigl
\{ \mu[B]\dvtx  B \in\mathcal{X}, B \subset A\bigr\};
\nonumber\\[-8pt]\\[-8pt]\nonumber
\mu^*[A] &=& \inf\bigl\{\mu[B]\dvtx  B \in\mathcal{X}, B \supset A\bigr\} = \min\bigl
\{ \mu[B]\dvtx  B \in\mathcal{X}, B \supset A\bigr\},
\end{eqnarray}
respectively.
\end{defn}

In (\ref{eqouter}), for example, the minimum is attained since the
intersection of the events $(B_n)_{n \in\N}$ satisfying $B_n\supset
A$ and $\mu^*[B_n] \leq\mu^*[A] + 1/n$ is again in the sigma algebra
$\mathcal{X}$.

%
\begin{lem}[{[\citet{Bierlein1962}, Satz~1A]}] \label{lBierlein}
Let $(X, \mathcal{X}, \mu)$ be a probability space and let $A \subset
X$. There exists an extension $\nu$ of $\mu$ to $\mathcal{X} \vee
\sigma\{A\}$ such that $\nu[A] = \mu^*[A]$.
\end{lem}

\begin{pf}
Observe that there exists an event $\widehat{A} \in\mathcal{X}$ so
that $\widehat{A} \supset A$ and $\widehat{A}{}^c \subset A^c$ with
$\mu^*[A] = \mu[\widehat{A}]$ and $\mu_*[A^c] = \mu[\widehat
{A}{}^c]$, where we denote complements by the superscript $^c$.
Moreover, we have
\[
\mathcal{X} \vee\sigma\{A\} = \bigl\{ (A \cap B_1) \cup
\bigl(A^c \cap B_2\bigr)\dvtx  B_1,
B_2 \in\F \bigr\}.
\]
Now, for any set $(A \cap B_1) \cup(A^c \cap B_2) \in \mathcal{X}
\vee\sigma\{A\}$ define
\[
\nu \bigl[(A \cap B_1) \cup\bigl(A^c \cap
B_2\bigr) \bigr] = \mu [\widehat{A} \cap B_1 ] + \mu
\bigl[\widehat{A} {}^c \cap B_2 \bigr].
\]
It is easy to see that $\nu$ is indeed a probability measure that
extends $\mu$ and satisfies $\nu(A) = \mu^*[A]$.
\end{pf}

Next, we state Parthasarathy's extension theorem.

%
\begin{teo}[{[\citet{Pa}, Theorem~V.4.1]}]\label{tparthasaratyextension}
Let $X$ be a set equipped with a standard system $(\X_n)_{n \in\N}$.
Let $(\mu_n)_{n \in\N}$ be a consistent family of probability
measures on $(\X_n)_{n \in\N}$, that is, $\mu_{n+1}\mid _{\X_n} = \mu
_n$ for all $n \in\N$. Then there exists a unique probability measure
on $\bigvee_{n \in\N} \X_n$, such that $\mu\mid _{\X_n} = \mu_n$ for
all $n \in\N$.
\end{teo}

%
\begin{teo}\label{tlubinextension}
Let $(X, \X)$ be a standard Borel space, and let $\mathcal{W}\subset
\X$ be a countably generated sigma algebra. Let $\mu$ be a
probability measure on $\mathcal{W}$. Then there exists a probability
measure $\nu$ on $\X$, such that $\nu\mid _{\mathcal{W}} = \mu$. The
extension $\nu$ is unique if and only if the sigma algebra $\X$ is
contained in the completion, with respect to the probability measure
$\mu$, of the sigma algebra $\mathcal{W}$.
\end{teo}

\begin{pf}
Let $(Y, \B(Y))$ be a Polish space along with its Borel sigma algebra,
and let $\varphi\colon\X\rightarrow\B(Y)$ be a $\sigma
$-isomorphism between $\X$ and $\B(Y)$. Define $\G= \{ \varphi(A)\dvtx
A \in\mathcal{W}\} \subset\B(Y)$. It is not hard to check that $\G
$ is a sigma algebra, that $\G$ is countably generated, and that
$\mathcal{W} = \varphi^{-1}(\G)$. Define the measure $m = \mu\circ
\varphi^{-1}$ on $\G$. If we can extend $m$ to a measure $n$ on $\B
(Y)$, then the proof is complete, because then we can set $\nu= n
\circ\varphi$.

The existence of an extension of $m$ from $\G$ to $\B(Y)$ is shown,
for example, in Theorem~5 in \citet{Lubin1974}. The result in
\citet{Lubin1974} is formulated for Blackwell spaces rather than
Polish spaces. However, each Polish space is a Blackwell space, as
defined in \citet{Lubin1974}.

The uniqueness result follows from Theorem~2 in \citet
{Ascherl1977}. We only need to show that if $A \in\X$, and if
$\widetilde{\mu}$ is an extension of $\mu$ to $\sigma(\mathcal{W}
\cup\{A\})$, then there exists an extension $\widetilde{\nu}$ of
$\mu$ on $\X$, such that $\widetilde{\nu}\mid _{\sigma(\mathcal{W}
\cup\{A\})} = \widetilde{\mu}$. However, the existence of such an
extension is an immediate consequence of the existence result since
$\sigma(\mathcal{W} \cup\{A\})$ is again countably generated.
\end{pf}

\section{Properties of the canonical path space}\label{sec10}\label{Apathspace}

We now collect some properties of the path space $(\Omega,\F)$ of
Assumption~\ref{assP}.

%
\begin{lem}[{[\citet{F1972}, Appendix]}] \label{lCanProps}
Under Assumption~\ref{assP}, we have the following statements:
\begin{longlist}[4.]
\item[1.] The probability space $(\Omega, \F)$ is standard Borel.
\item[2.] Let $\rho$ denote a stopping time. Then the sigma algebra $\F
_{\rho-}$ is countably generated. [Recall the definition of $\F_{\rho
-}$ in (\ref{eFrho-definition}).]
\item[3.] Let $(\rho_n)_{n\in\N}$ be a nondecreasing sequence of $(\F
_t)_{t \geq0}$-stopping times. Then the filtration $(\F_{\rho
_n-})_{n \in\N}$ is a standard system.
\item[4.] The set identity $\F_{\zeta-} = \F$ holds.
\end{longlist}
\end{lem}
\begin{pf}
Meyer has shown [see page~100 in \citet{Dellacherie1969}] that
there exists a Polish space $Y$, with Borel sigma algebra $\B(Y)$,
such that $(\Omega, \F)$ is isomorphic to $(Y,\B(Y))$. This implies
the first part of the statement.

For the second statement, note that the set equality
\[
\{\rho> t\} = \bigcup_{q \in\Q\dvtx  q > t} \{\rho> q\}
\]
holds and $\F^0_t$ is countably generated for each $t \geq0$; see
also the observations after Lemma~\ref{Lcountably}. Thus, if
$(A^{(t)}_m)_{m \in\N}$ is a countable generating system for $\F
^0_t$ for each $t \geq0$, then
\[
\F_{\rho-} = \sigma \bigl(A^{(0)}_m,
A^{(q)}_m \cap\{ \rho> q \}\dvtx  q \in\Q\cap[0,\infty), m
\in\N \bigr).
\]

For the third statement, property~(i) in the definition of a standard
system follows directly from Lemma~\ref{Lcountably}. Property~(ii) is
easy to check.

For the fourth statement, it suffices to show that for each Borel set
$B \subset E \cup\{\Delta\}$ and each $t \ge0$ we have $\{\omega\dvtx
\omega(t) \in B\} \in\F_{\zeta-}$. Toward this end, note that
\[
\bigl\{\omega(t) \in B\bigr\} = \bigl(\bigl\{\omega(t) \in B\bigr\} \cap\{\zeta
\le t\} \bigr) \cup \bigl(\bigl\{\omega(t) \in B\bigr\} \cap\{\zeta> t\} \bigr).
\]
The first event on the right-hand side equals $\{\zeta\le t\}$ if $B$
contains $\Delta$, and it is the empty set otherwise. Therefore, that
event is in $\F_{\zeta-}$. The second event is in $\F_{\zeta-}$ by
definition.
\end{pf}

%
\begin{teo} \label{TC2}
Under Assumption~\ref{assP}, let $\rho$ be a stopping time and
$\mu$ a probability measure on $(\Omega, \F_{\rho-})$. Then the
measure $\mu$ can be extended to a probability measure $\nu$ on
$(\Omega, \F)$ such that $\nu\mid _{\F_{\rho-}} = \mu$.
Moreover, that extension $\nu$ is unique if and only if the set
$\{\rho< \zeta\}$ is $\mu$-negligible.
\end{teo}

\begin{pf}
Since $(\Omega, \F)$ is a standard Borel space and $\F_{\rho-}$ is
countably generated by Lemma~\ref{lCanProps}, Theorem~\ref
{tlubinextension} implies the existence of an extension $\nu$ of $\mu
$ from $\F_{\rho-}$ to $\F$.
Moreover, the extension $\nu$ is unique if and only if $\F_{\zeta-}
= \F\subset\F_{\rho-}^{\mu}$, where\vspace*{1pt} the first equality follows
from Lemma~\ref{lCanProps} and where the completion of a sigma algebra
$\X$ with respect to the probability measure $\mu$ is denoted by $\X
^\mu$.

Assume now that the set $\{\rho< \zeta\}$ is $\mu$-negligible. To
prove that $\F_{\zeta-} \subset\F_{\rho-}^{\mu}$, it suffices to
show that $A \cap\{\zeta> t\} \in\F_{\rho-}^{\mu}$ for all $A \in
\F_t$ and $t \ge0$. Toward this end, note that
\[
A \cap\{ \zeta> t\} = \bigl(A \cap\{ \zeta> t\} \cap\{\rho> t\} \bigr) \cup
\bigl( A \cap\{ \zeta> t\} \cap\{\rho\le t\} \bigr)
\]
for all $A \in\F_t$ and $t \geq0$.
The first event on the right-hand side is in $\F_{\rho-}$ since $\{
\zeta> t\} \in\F_t$ holds for all $t \geq0$. The second event on
the right-hand side is contained in the $\mu$-negligible set $\{\rho
< \zeta\}$ and, therefore, it is an element of $\F^{\mu}_{\rho-}$.

For the reverse direction, assume that the set $\{\rho< \zeta\}$ is
not $\mu$-negligible,
which implies that its outer content is strictly positive, that is,
$\mu^*[\rho< \zeta]>0$.
By Lemma~\ref{lBierlein}, there exists an extension $\widehat{\nu}$
from $\F_{\rho-}$ to $\F_{\rho-} \vee\sigma(\{\rho< \zeta\})$,
such that $\widehat{\nu}[\rho< \zeta] > 0$. Since $\F_{\rho-}
\vee\sigma(\{\rho< \zeta\})$ is countably generated, Theorem~\ref
{tlubinextension} yields an extension $\nu$ of $\widehat{\nu}$ from
$\F_{\rho-} \vee\sigma(\{\rho< \zeta\})$ to $\F$.

Next, fix a sufficiently large $n \in\N$ so that the stopping time
$\widetilde{\rho} = \rho+ 2/n$ satisfies $\nu[\widetilde{\rho} <
\zeta]>0$ and define the measurable function $\phi\colon\Omega
\rightarrow\Omega$ by $\phi(\omega)(t) = \omega(t)$ for all $t \in
[0,\widetilde{\rho})$ and $\phi(\omega)(t) = \Delta$ for all $t
\in[\widetilde{\rho},\infty)$. This mapping introduces a new
probability measure $\widetilde{\nu} = \nu\circ\phi^{-1}$, such that
$\widetilde{\nu}[\widetilde{\rho} < \zeta]=0$.

Observe furthermore that $\rho+ 1/n$ is an $(\F^0_t)_{t \geq
0}$-stopping time and that $\F_{\rho-} \subset\F^0_{\rho+
1/n}=\sigma\{\omega(t \wedge(\rho+ 1/n))\dvtx  t \ge0\}$, where the
last equality can be shown, for example, as in Lemma~1.3.3 in
\citet{SVmulti}. We conclude that $\widehat{\nu}\mid _{\F_{\rho
-}} = \nu\mid _{\F_{\rho-}}$ but $\widetilde{\nu}[\widetilde{\rho} <
\zeta] = 0 <\nu[\widetilde{\rho} < \zeta]$, and thus the extension
is not unique.
\end{pf}

%
\begin{lem}\label{lomegastaysincountablymanyx}
Under Assumption~\ref{assP}, let $\mu$ be a probability measure
on $(\Omega,\F)$. Then the set
%
\begin{equation}
\label{exwhereomegastayswithpositiveproba} A = \bigl\{x \in E\dvtx \mu\bigl[\bigl\{ \omega\dvtx  \exists t \ge0,
\varepsilon> 0,\mbox{ s.t. } \omega(s) = x\mbox{ for all } s \in[t,t+\varepsilon)\bigr\}\bigr] > 0 \bigr\}\hspace*{-25pt}
\end{equation}
is at most countable.
\end{lem}

\begin{pf}
The proof is an adaption of the arguments in Lemma~3.7.7 of \citet
{EKMarkovian}. Let $T,\varepsilon,\delta> 0$. We claim that the set
\[
A_{T, \varepsilon, \delta} = \bigl\{ x \in E\dvtx  \mu\bigl[\bigl\{\omega\dvtx  \exists t \in[0,T]\mbox{
s.t. } \omega(s) = x\mbox{ for all } s \in [t, t + \varepsilon)\bigr\}\bigr]
> \delta \bigr\}
\]
is finite. If the claim holds, then the set inclusion $A \subset
\bigcup_{n \in\N} A_{n, 1/n, 1/n}$ yields the statement.

To prove this claim, assume that there exists an infinite sequence
$(x_n)_{n \in\N}$ in $A_{T, \varepsilon, \delta}$, where $x_n \neq
x_m$ whenever $n \neq m$ for all $n, m \in\N$. Then we have
\begin{eqnarray*}
&& \mu \biggl[\bigcap_{k \in\N} \bigcup
_{n\geq k} \bigl\{\omega\dvtx  \exists t \in[0,T]\mbox{ s.t. }
\omega(s) = x_n\mbox{ for all } s \in[t, t + \varepsilon )\bigr\}
\biggr]
\\
&&\qquad = \lim_{k \uparrow\infty}\mu \biggl[\bigcup
_{n\geq
k} \bigl\{\omega\dvtx  \exists t \in[0,T] \mbox{ s.t. }
\omega(s) = x_n \mbox { for all } s \in[t, t + \varepsilon )\bigr
\} \biggr] > \delta.
\end{eqnarray*}
However, for each $\omega\in\Omega$ there are at most $\lfloor(T +
\varepsilon) / \varepsilon\rfloor$ values of $x \in E$ for which
there exists $t \in[0,T]$ such that $\omega(s) = x$ for all $s \in
[t, t+\varepsilon)$, a contradiction. Here, $\lfloor\cdot\rfloor$
denotes the Gauss bracket.
\end{pf}

\section{Supermartingales as Fatou limits of~martingales}\label{sec11} \label{AFatou}

Similar techniques as developed in Section~\ref{SSapprox} allow us to
show that each nonnegative $P$-supermartingale is the Fatou limit of a
sequence of uniformly integrable martingales, provided that the
probability space supports a Brownian motion. Toward this end, recall
the definition of Fatou convergence.

%
\begin{defn}
A sequence of processes $(G^{(n)})_{n \in\N}$ with $G^{(n)} =
(G^{(n)}_t)_{t \ge0}$ \emph{Fatou converges} to a process
$G = (G_t)_{t \geq0}$ if there exists a dense subset $\mathcal{T}$ of
$[0,\infty)$, such that
\[
\liminf_{s \downarrow t, s \in\mathcal{T}} \Bigl(\liminf_{n
\uparrow\infty}
G^{(n)}_s \Bigr) = \limsup_{s \downarrow t, s \in
\mathcal{T}} \Bigl(
\limsup_{n \uparrow\infty} G^{(n)}_s \Bigr) =
G_t
\]
holds for all $t \ge0$ almost surely.
\end{defn}

%
\begin{teo} \label{TFatou}
Under Assumption~\ref{assB}, let the $P$-supermartingale $Z $
have Doob--Meyer decomposition $Z = M + D$. Then there exists a sequence
of uniformly integrable nonnegative martingales $(Z^{(m)})_{m \in\N}$
with $Z^{(m)} = \break (Z_t^{(m)})_{t \geq0}$ such that there exists a dense
subset $\mathcal{T}$ of $[0,\infty)$, whose complement is a Lebesgue
null set, such that $\lim_{m \uparrow\infty} Z^{(m)}_t = M_t +
D_{t-}$ for all $t \in\mathcal{T}$ almost surely. In particular,
$(Z^{(m)})_{m \in\N}$ Fatou converges to the $P$-supermartingale $Z$.
\end{teo}

\begin{pf}
Let us start by approximating the left-continuous process $D_-$ by a
sequence $(N^{(m)})_{m \in\N}$ of local martingales with $N^{(m)} =
(N^{(m)}_t)_{t \geq0}$, similarly to Lemmas~\ref{lapproximatedecreasingbysimpleprocesses} and \ref
{lapproximatesimplebylocalmartingales}. Toward this end, fix $m \in\N
$, set $N^{(m)}_0 = D_0$, and keep $N^{(m)}$ constant up to time
$2^{-m}$. Initiate then a ``mass loss phase'' such that $N^{(m)}$
fluctuates on $(2^{-m}, 2^{-m} + 2^{-3m})$, until it reaches
$D_{2^{-m}}$ at time $2^{-m} + 2^{-3m}$.
Now, $N^{(m)}$ stays constant on $[2^{-m} + 2^{-3m}, 2\times2^{-m}]$,
until at time $2 \times2^{-m}$ we initiate the next mass loss phase,
so that $N^{(m)}$ fluctuates again on an interval of length $2^{-3m}$,
until it reaches $D_{2 \times2^{-m}}$. For $k=3,\ldots, m2^{m}$ we
continue by having mass loss phases on the interval $(k 2^{-m}, k
2^{-m} + 2^{-3m})$, at the end of which we reach $D_{m2^{-m}}$. From
time $m + 2^{-3m}$ on, the process $N^{(m)}$ stays constant.

Next, set
%
\begin{equation}
\label{eqFatouS} \mathcal{S} = \bigcap_{n=1}^\infty
\bigcup_{m=n}^\infty\bigcup
_{k=0}^{m 2^m} \bigl(k 2^{-m}, k
2^{-m} + 2^{-3m}\bigr)
\end{equation}
and note that $\mathcal{T} = [0,\infty) \setminus\mathcal{S}$ is a
dense subset of $[0,\infty)$ since the set $\mathcal{S}$ has Lebesgue
measure zero. The set $\mathcal{S} $ can be interpreted as the set of
all points that lie in infinitely many ``mass loss intervals.'' Now,
fix $t \in\mathcal{T}$ and note that $(N_t^{(m)})_{m \in\N}$ is
eventually a nonincreasing sequence with $\lim_{m \uparrow\infty}
N^{(m)}_t = D_{t-}$.
Moreover, observe that $N^{(m)}$ almost surely attains a finite maximal
absolute value, and therefore there exists a constant $K_m>0$ such that
for $\rho_{m} = \inf\{t \ge0\dvtx  |N^{(m)}_t| \ge K_m\}$ we have $P[\rho
_{m} < \infty] < 2^{-m}$ for each $m \in\N$.

Next, define the uniformly integrable nonnegative martingales
\[
Z^{(m)} = M^{\tau^M_m \wedge\rho_m} + \bigl(N^{(m)}
\bigr)^{\tau^M_m \wedge
\rho_m},
\]
for all $m \in\N$, where $(\tau^M_m)_{m \in\N}$ is a localizing
sequence for the $P$-local martingale~$M$. An application of the
Borel--Cantelli lemma then yields that $\lim_{m \uparrow\infty}
Z^{(m)}_t = M_t + D_{t-}$ for all $t \in\mathcal{T}$ almost surely.
The Fatou convergence follows directly from the right-continuity of the
$P$-supermartingale $Z$.
\end{pf}

%
\begin{rmk}
In the proof of Theorem~\ref{TFatou}, we used the fact that $\mathcal
{S}$, the set of all points that lie in infinitely many ``mass loss
intervals'' given in (\ref{eqFatouS}), has Lebesgue measure zero. One
might suspect that the set $\mathcal{S}$ is countable or even empty.
However, this is not true. There exists a suitable strictly increasing
sequence $(m_n)_{n \in\N}$ with $m_n \in\N$ for each $n \in\N$
such that the map $\varphi\colon\{0,1\}^\N\rightarrow[0,2)$,
$(a_n)_{n \in\N} \mapsto\sum_{n =1}^\infty a_n 2^{-m_n}$ satisfies
$\varphi((a_n)_{n \in\N}) \in\mathcal{S}$ for each sequence
$(a_n)_{n \in\N}$ that satisfies $\sum_{n=1}^\infty a_n = \infty$.
As a consequence, $\mathcal{S}$ is a Cantor-like set, in the sense
that it is an uncountable Lebesgue-null set.
\end{rmk}

\section{Finitely additive measures on~the~dyadic~filtration}\label{sec12} \label{AE}
Here, we show that on the unit interval equipped with the dyadic
filtration, no finitely additive probability measure is uniquely
determined by its values on the dyadic algebra generating the Borel
$\sigma$-algebra.

Let $\Omega= (0,1]$ be equipped with the Borel sigma field $\F= \B
(\Omega)$, let
\[
\F_n = \sigma\bigl(\bigl(k2^{-n},(k+1)2^{-n}\bigr]\dvtx  0 \le
k \le2^n - 1\bigr)
\]
for all $n \in\N$, and let $P$ be a finitely additive probability
measure on $(\Omega,\F)$. We will construct a finitely additive
measure $\widetilde{P} \neq P$ such that $\widetilde{P}$ agrees with
$P$ on the algebra $\bigcup_{n \in\N} \F_n$.

Let $A = \{x_m\dvtx  m \in\N\}$ denote a countable dense subset of
$(0,1]$, such that $P[A] < 1$. Such a set has to exist since there are
disjoint countable dense subsets of $(0,1]$, for example, $\Q\cap
(0,1]$ and $(\pi+ \Q) \cap(0,1]$.
Next, for each $k,n \in\N$ such that $ k \le2^n - 1$ we define
$y^{(n)}_k = x_{m(n,k)}$, where $m(n,k)$ is the smallest integer $m$
with $x_m \in A \cap(k 2^{-n}, (k+1)2^{-n}]$, and define the set
function $\widetilde{P}{}^{(n)}\colon\F\rightarrow[0,1]$ by
\[
\widetilde{P}{}^{(n)}[B] = \sum_{k=0}^{2^n-1}
P \bigl[\bigl(k2^{-n}, (k+1)2^{-n}\bigr] \bigr] \1_B
\bigl(y^{(n)}_k \bigr)
\]
for all $B \in\F$. By construction, we have $\widetilde
{P}{}^{(n)}\mid _{\F_n} = P\mid _{\F_n}$. Note that $\widetilde{P}{}^{(n)}$ is a
countably additive probability measure and, therefore,
\[
\widetilde{Q} = \sum_{n=0}^{\infty}
2^{-n-1} \widetilde{P}{}^{(n)}
\]
is a countably additive probability measure such that $\widetilde
{P}{}^{(n)}$ is absolutely continuous with respect to $\widetilde{Q}$
for each $n \in\N$. Thus, for each $n \in\N$, the set function
$\widetilde{P}{}^{(n)}$ can be identified with an element of the unit
ball of $(L^\infty(\widetilde{Q}))^*$ and,\vspace*{1pt} therefore, the
Banach--Alaoglu theorem implies the existence of a subsequence
$(\widetilde{P}{}^{(n_k)})_{k \in\N}$ that converges in $(L^\infty
(\widetilde{Q}))^*$ to a finitely additive probability measure
$\widetilde{P}$. By construction, $\widetilde{P}|_{\F_n} = P|_{\F
_n}$ for all $n \in\N$, and $\widetilde{P}[A] = 1$, whereas $P[A] < 1$.

For instance, we could take $P$ as the Lebesgue measure. Carath\'
eodory's extension theorem then implies the uniqueness of the extension
of $P$ from $\bigcup_{n \in\N} \F_n$ to~$\F$. However, as we just
illustrated, this extension is only unique among the sigma additive
measures, not among the larger class of finitely additive measures.

\section{Alternative proof of Lemma~\texorpdfstring{\protect\ref{1111}}{5.12}}\label{sec13}\label{Aalternative}
We here provide a proof of Lemma~\ref{1111} that is
more ``constructive'' than the one in Section~\ref{SSnonuniqueness}. It
relies on the following lemma.

\begin{lem}\label{lregularpartimpliesmasslessthanone}
Let $Q = Q^r + Q^s$ be a finitely additive probability measure, where
$Q^r$ is the regular part and $Q^s$ is the singular part. If $Q^r \neq
0$ is absolutely continuous with respect to $P$, then there exists
$\varepsilon> 0$, such that $Q[A] < 1$ for each event $A \in\F$ with
$P[A]<\varepsilon$.
\end{lem}

\begin{pf}
Assume that there exists no such $\varepsilon> 0$ as in the statement.
Then, for each $n \in\N$, there exists an event $A_n \in\F$ such
that $P[A_n] \le1/n$ but $Q[A_n] = 1$. The dominated convergence
theorem then implies that
\[
1 = \lim_{n \uparrow\infty} Q[A_n] = \lim_{n \uparrow\infty}
\biggl( \E_P \biggl[\frac{\dd Q^r}{\dd P} \1_{A_n} \biggr] +
Q^s[A_n] \biggr)= \lim_{n\uparrow\infty}
Q^s[A_n],
\]
so that $Q^s[\Omega] = 1 = Q[\Omega]$, a contradiction to $Q^r \neq0$.
\end{pf}

\begin{pf*}{Proof of Lemma~\ref{1111}}
Define the nonnegative martingale $Z^{(1)} = (Z^{(1)}_t)_{t
\geq0}$ as the right-continuous modification of the uniformly
integrable $P$-martingale $(\E_{P} [Z_\infty|\F_t])_{t \geq0}$ and
the $P$-potential $Z^{(2)} = (Z^{(2)}_t)_{t \geq0}$ by $Z^{(2)} = Z- Z^{(1)}$.
There exist a F\"ollmer finitely additive measure $Q^{(2)}$ for the
$P$-potential $Z^{(2)}/\E_P[Z^{(2)}_0]$ (assuming that $\E
_P[Z^{(2)}_0]>0$, otherwise set the measure to zero) and a F\"ollmer
countably additive measure $Q^{(1)}$ for the uniformly integrable
$P$-martingale
$Z^{(1)}/\E_P[Z^{(1)}_0]$, yielding a F\"ollmer finitely additive
measure $Q =\E_P[Z^{(1)}_0] Q^{(1)} + \E_P[Z^{(2)}_0] Q^{(2)}$ for
the $P$-supermartingale $Z$. Note that $Q^r>0$ since $Q^r(\dd\omega)
= Z_\infty(\omega) P(\dd\omega) \neq0$.

Next, for all $n \in\N$ and $\varepsilon>0$,
choose $k_{n, \varepsilon} \in\N$, such that $P[A_{n,\varepsilon}]
\le\varepsilon2^{-n}$, where
\[
A_{n,\varepsilon} =\bigl\{ |W_t - W_n| < 1 \mbox{ for
all } t \in[n, k_{n, \varepsilon}]\bigr\}.
\]
Set $A_\varepsilon= \bigcup_{n \in\N} A_{n,\varepsilon} $ and note
that $P[A_\varepsilon] \le\varepsilon$. By Lemma~\ref{lregularpartimpliesmasslessthanone} we have $Q[A_\varepsilon] < 1$
for some $\varepsilon> 0$. Fix such an $\varepsilon$ and recall
Proposition~\ref{PapproxofZ}. We now replace the nonnegative
$P$-martingale $L^{(i,j,k,m,n)}$ in (\ref{eq59convergence})
by the process $\widehat{L}{}^{(i,j,k,m,n)} = (\widehat
{L}{}^{(i,j,k,m,n)}_t)_{t \geq0}$, given by
\[
\widehat{L}{}^{(i,j,k,m,n)}_t = {L}^{(i,j,k,m,n)}_{t \wedge n} \E
_P \biggl[ \frac{\1_{A_{n,\varepsilon}}}{\E_P [\1_{A_{n,\varepsilon}}|\F
_n ]}\bigg| \F_t \biggr]
\]
for all $t \geq0$, for each $i,j,k,m,n \in\N$. Note that (\ref
{eq59convergence}) holds after this replacement and the same argument
as in the existence proof of Theorem~\ref{Tapprox} yields a F\"ollmer
finitely additive measure $\widehat{Q}$, but now we have $\widehat
{Q}[A_\varepsilon] = 1$, which completes the proof.
\end{pf*}
This proof illustrates that it is always possible to ``destroy'' some
possibly remaining regular part of a F\"ollmer finitely additive
measure at infinity.
\end{appendix}

\section*{Acknowledgements}
We are grateful to Sara Biagini, Zhenyu Cui, Ioannis Karatzas, Kostas
Kardaras, Martin Larsson, Irina Penner, Walter Schachermayer, Bill
Sudderth and Mikhail Urusov for helpful discussions on the subject
matter. We thank the anonymous referee for their insightful and helpful
comments and for providing us with the references for \citet
{Moy1953} and \citet{Azema1976}.



%

\printaddresses
\end{document}